\documentclass[DIV=13,abstract=true,paper=a4,fontsize=11pt,parskip=half]{scrartcl}

\RequirePackage{booktabs}
\RequirePackage{dsfont}
\RequirePackage{microtype}
\usepackage{lmodern}
\usepackage{inputenc}
\usepackage[T1]{fontenc}
\usepackage[USenglish]{babel}
\usepackage{latexsym}
\usepackage{amsmath,amsfonts,amssymb,amsthm,amsopn}
\usepackage[centercolon]{mathtools}
\usepackage{fixmath}
\usepackage{braket}
\usepackage{graphicx}
\graphicspath{{pics/}}
\makeatletter
\def\input@path{{pics/}}
\makeatother
\usepackage{epstopdf}
\usepackage[svgnames]{xcolor}
\usepackage{tikz}
\usepackage{pgfplots,pgfplotstable}
\usepackage[algo2e,ruled,linesnumbered]{algorithm2e}
\usepackage{subcaption}
\usepackage{pdflscape}
\pgfplotsset{compat=newest}
\pgfplotsset{scaled y ticks=false}
\usepgfplotslibrary{groupplots}
\usepgfplotslibrary{dateplot}
\pgfplotsset{plot coordinates/math parser=false}
\usepackage{nameref}
\usepackage{xspace}
\usepackage{rcs}
\usepackage{enumerate}
\usepackage{url}
\usepgfplotslibrary{external} 
\usetikzlibrary{external}
\usepgfplotslibrary{statistics}
\usetikzlibrary{pgfplots.statistics}
\usepackage{colortbl}
\usepackage{multirow}
\usepackage{array}
\usepackage{makecell}
\usepackage{enumerate}
\usepackage[pdftex,pagebackref,bookmarks=true,pdfpagelabels=true,plainpages=false,hypertexnames=false,bookmarksnumbered=true,bookmarksopen=true,pdfstartview=Fit,colorlinks,linkcolor=DarkBlue,citecolor=DarkGreen,urlcolor=DarkRed]{hyperref}
\usepackage[nameinlink,capitalise]{cleveref}

\SetKw{KwTerminate}{terminate}

\SetCommentSty{mycommfont}
\crefformat{equation}{(#2#1#3)}
\crefname{assumption}{Assumption}{Assumptions}
\Crefname{assumption}{Assumption}{Assumptions}
\crefname{figure}{Figure}{Figures}
\theoremstyle{plain}
\newtheorem{theorem}{Theorem}[section]
\newtheorem{lemma}[theorem]{Lemma}
\newtheorem{corollary}[theorem]{Corollary}
\theoremstyle{definition}
\newtheorem{definition}[theorem]{Definition}
\newtheorem{assumption}[theorem]{Assumption}

\theoremstyle{remark}
\newtheorem{remark}[theorem]{Remark}

\providecommand{\MSC}[1]{\textbf{Mathematics Subject Classification (2020).} #1}

\newdimen\fwd

\newcommand{\norm}[1]{\ensuremath{\lVert #1\rVert}}
\newcommand{\xopt}{\ensuremath{x^\ast}}
\newcommand{\BB}{Barzilai--Borwein}
\newcommand{\BBB}{BB}
\newcommand{\WP}{Wolfe--Powell}
\newcommand{\MT}{Moré--Thuente}
\newcommand{\Ballop}[1]{\ensuremath{\mathbb{B}_{#1}}}
\newcommand{\Lin}{\ensuremath{\mathcal{L}}}
\newcommand{\LinPDX}{\ensuremath{\mathcal{L}_+(\CX)}}
\newcommand{\brk}{\ensuremath{break}}
\newcommand{\outp}{\ensuremath{output}}
\newcommand{\CP}{\#{\cal P}}
\newcommand{\CI}{{\cal I}}
\newcommand{\fopt}{{f^\ast}}
\newcommand{\CX}{{\cal X}}
\newcommand{\nbhd}{{\cal N}}
\newcommand{\R}{\mathbb{R}} 
\newcommand{\N}{\mathbb{N}} 
\newcommand{\Hy}{-}
\newcommand{\Hs}{+}
\newcommand{\LBFGS}{\mbox{L-BFGS}}

\title{A globalization of \LBFGS~and the \BB~method for nonconvex unconstrained optimization}

\author{   Florian Mannel\thanks{Institute of Mathematics and Image Computing, University of Lübeck, Maria-Goeppert-Straße 3, 23562 
		Lübeck, Germany (\href{mailto:florian.mannel@uni-luebeck.de}{florian.mannel@uni-luebeck.de})}\hspace*{0.05cm}
	\href{https://orcid.org/0000-0001-9042-0428}{\includegraphics[height=.35cm]{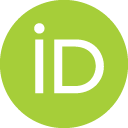}}   }

\date{Preprint, \today}

\ifpdf
\hypersetup{
	pdftitle={A globalization of \LBFGS~and the \BB~method for nonconvex unconstrained optimization},
	pdfsubject={A globally convergent version of the \LBFGS- and \BB-method},
	pdfauthor={F. Mannel},
	pdfkeywords={nonconvex optimization, cautious updates, limited memory quasi-Newton methods, BFGS methods, Barzilai--Borwein methods, convergence analysis, Hilbert space}
}
\fi

\begin{document}

\maketitle

\begin{abstract}
	We present a modified limited memory BFGS (\LBFGS) method that converges globally and linearly for nonconvex objective functions. 
	Its distinguishing feature is that it turns into \LBFGS~if the iterates cluster at a point near which the objective is strongly convex with Lipschitz gradients, thereby inheriting the outstanding effectiveness of the classical method. 
	These strong convergence guarantees are enabled by a novel form of cautious updating, where, among others, 
	it is decided anew in each iteration which of the stored pairs are used for updating and which ones are skipped. 
	In particular, this yields the first modification of cautious updating for which all cluster points are stationary 
	while the spectrum of the L-BFGS operator is not permanently restricted,
	and this holds without Lipschitz continuity of the gradient. 
	In fact, for \WP~line searches we show that continuity of the gradient is sufficient for global convergence, 
	which extends to other descent methods. 
	Since we allow the memory size to be zero in the globalized \LBFGS~method, we also obtain a new globalization of the \BB~spectral gradient~(\BBB) method.
	The convergence analysis is developed in Hilbert space under 
	comparably weak assumptions and covers Armijo and \WP~line searches.
	We illustrate the theoretical findings with numerical experiments. The experiments indicate that if one of the parameters of the cautious updating is chosen sufficiently small, then the modified method agrees entirely with \LBFGS/\BBB. We also discuss this in the theoretical part. 
	An implementation of the new method is available on \textsc{arXiv}.
\end{abstract}

\begin{keywords}
	L-BFGS, Barzilai--Borwein methods, cautious updates, nonconvex optimization, global convergence, linear convergence, Hilbert space
\end{keywords}

\MSC{65K05, 65K10, 90C06, 90C26, 90C30, 90C48, 90C53}



\section{Introduction}\label{sec_intro}

\LBFGS~\cite{N80,LN89,BNS94} is one of the most popular methods for large-scale unconstrained optimization problems. 
Among others, it is used in geosciences \cite{LBLPT21}, image registration \cite{AMM23}, computer-generated holography \cite{SKYWW23} and machine learning \cite{GRB20}. Despite the maturity of \LBFGS, there are still fundamental open questions, for instance
if \LBFGS~converges globally on nonconvex problems. 

In this paper, we study a globally convergent version of \LBFGS~for the problem
\begin{equation}\label{eq_prob}\tag{P}
	\min_{x\in\CX} \, f(x),
\end{equation}
where $f:\CX\rightarrow\R$ is continuously differentiable and bounded below, and $\CX$ is a Hilbert space.
The convergence analysis includes the possibility that the memory size in \LBFGS~is zero, in which case the new method becomes a globalized \emph{\BB~(BB) method} \cite{BB88,R97,DL02,DF05,DHSZ06}, also called \emph{spectral gradient method}.  
Since they are suited for large-scale problems, \BBB-type methods have recently received renewed interest \cite{DKPS15,BDH19,DHL19,AK20,LMP21,AK22}.

To ensure global convergence on nonconvex problems, various modified \LBFGS~methods have been developed. 
To the best of our knowledge, however, they all suffer from at least one of the following issues:
\begin{itemize}
	\item Global convergence is often established as $\liminf_{k\to\infty}\norm{\nabla f(x_k)}=0$, e.g., \cite{KD15,YWS20,YZZ22}; this does not guarantee that cluster points of $(x_k)$ are stationary. 
	\item If rate of convergence results are provided, they frequently rely on assumptions whose satisfaction is unclear in the nonconvex case; examples include 
	convergence of $(x_k)$ in \cite[Section~3]{DL02} and the existence of $c>0$ such that 
	$y_k^T s_k > c \norm{y_k}\norm{s_k}$ for all $k$ in \cite[Theorem~3.2]{BBEM19}.
	In particular, these properties may not hold even if $\liminf_{k\to\infty}\norm{\nabla f(x_k)}=0$, $(x_k)\subset\R^N$ is bounded, and $f$ is $C^\infty$ with a single stationary point that satisfies sufficient optimality conditions and that is the unique global minimizer. 
	\item Some methods rely on beforehand knowledge of the local modulus of strong convexity. 
	This comprises, for instance, methods that 
	skip the update if $y_k^T s_k \geq \mu \norm{s_k}^2$ is violated, where $\mu>0$ is chosen at the beginning of the algorithm, e.g., \cite{BJRT22,AMM23}. 
	Clearly, if $f$ is strongly convex and $\mu$ is no larger than the modulus of strong convexity, all updates are carried out (as in \LBFGS). On the other hand, if $\mu$ is too large, many or even all updates may be skipped, which will usually diminish the efficiency significantly. In nonconvex situations this problem appears if $(x_k)$ converges to a point near which $f$ is strongly convex. In globalized \BBB~methods and \LBFGS~with 
	seed matrices $H_0^{(k)}=\gamma_k I$, this issue (also) arises when the step sizes, respectively, the scaling factors $(\gamma_k)$ are safeguarded away from zero.
	Ultimately, all these constructions permanently restrict the spectrum of the \LBFGS~operator $H_k$ (cf.~\Cref{lem_generalresultonnormbounds} for a proof of this fact). 
\end{itemize}

The method provided in this paper does not suffer from any of these issues, but instead comes with the following strong convergence guarantees:
\begin{enumerate}
	\item[1)] every cluster point is stationary, cf. \Cref{thm_clusterpointsarestationar};
	\item[2)] $\lim_{k\to\infty}\norm{\nabla f(x_k)}=0$ if $\nabla f$ is uniformly continuous in some level set, cf. \Cref{thm_globconv}; 
	\item[3)] if the iterates cluster at a point near which $f$ is strongly convex and near which $\nabla f$ is Lipschitz, then they converge to this point at a linear rate and the method agrees with classical \LBFGS~after a finite number of iterations, cf.~Theorems~\ref{thm_linconv} and \ref{thm_lbfgsmislbfgs}; \item[4)] under the assumptions of 3), all update pairs are eventually stored and applied, and any 
	$\gamma_k$ that lies between the two \BB~step sizes (see Definition~\ref{def_tau_B}) is eventually accepted, cf. \Cref{thm_lbfgsmislbfgs}.
\end{enumerate}
To the best of our knowledge, the new method is the only \LBFGS-type method that satisfies both 1) and 3). 
As a consequence of 3), the new method will often inherit the supreme efficiency of \LBFGS. 
In fact, our numerical experiments indicate that the new method agrees \emph{entirely} with \LBFGS/\BBB~if one of its algorithmic parameters is chosen sufficiently small. 
This may be regarded as an explanation why \LBFGS/\BBB~often converges for nonconvex problems, but it is also noteworthy since some authors report that cautious updating can degrade the performance, at least for BFGS, cf. \cite[p.18]{GR23}.
A discussion of this agreement is offered in Section~\ref{sec_transitiontoLBFGS}, particularly in \Cref{rem_LBFGSMagreeswithLBFGSforsmallconstants}. 

Many existing globalizations of the \BB~method also permanently restrict the spectrum of $H_k=H_k^{(0)}$, 
so the techniques and results in this paper are of interest in the context of \BBB~methods, too.
That being said, \emph{nonmonotone} line searches, which
are understood to be more effective for \BBB-type methods \cite{AK20}, are 
only addressed in the numerical experiments in Section~\ref{sec_experiments}, but not in the convergence analysis.

The convergence guarantees of the new method and the transition to classical \LBFGS~rely on a new form and a careful calibration of the \emph{cautious updates} introduced for BFGS by Li and Fukushima \cite{LiFu01}.
We discuss these points further once we have introduced the globalized \LBFGS~method in Section~\ref{sec_LBFGSM}, respectively, once we have introduced its calibration in \Cref{ass_linconv}, but let us stress that the new \LBFGS~method proposed in this paper is fully compatible with the techniques used in efficient numerical realizations of \LBFGS. In particular, it can still be implemented matrix free based on the well-known two-loop recursion and it has essentially the same costs per iteration as \LBFGS. 

We strive to use the weakest possible assumptions in the convergence analy\-sis.  
For instance, in \Cref{thm_linconv} we prove that the new method converges globally and linearly 
under purely \emph{local} assumptions except for the requirement that $f$ has to be bounded below. 
Specifically, $f$ needs to be continuously differentiable and bounded below, and $(x_k)$ has to have a cluster point near which $f$ is strongly convex and $\nabla f$ is Lipschitz.
In contrast, in the existing literature for \LBFGS-type methods it is frequently assumed that $\nabla f$ is Lipschitz continuous in the level set $\Omega$ associated to the initial point, that $\Omega$ is bounded and that $f$ is twice continuously differentiable; cf. \cite{LN89,ABGP14,KD15,BBEM19,BJRT22,TSY22}. 
Note, however, that if $f$ is actually twice continuously differentiable, then our assumption of strong convexity near a cluster point is equivalent to the Hessian of $f$ being positive definite at that cluster point. We avoid using this classical second order sufficient condition for optimality because we do not want to require the existence of second derivatives. 
The results of this work hold not only for the \WP~conditions but also for backtracking line searches based on the Armijo condition. 
To the best of our knowledge, it has not been proven before that the \WP~conditions imply global convergence if $\nabla f$ is merely continuous, cf. \Cref{thm_clusterpointsarestationar}. Since this result can be extended to other descent methods, it may be of interest beyond this work. 

Another improvement of existing results concerns the \emph{type} of linear convergence. The strongest available result for 
the original \LBFGS~method is the classical \cite[Theorem~7.1]{LN89},
where it is shown that $(f(x_k))$ converges q-linearly and $(x_k)$ converges r-linearly. We obtain the same rate for $(f(x_k))$,
but for $(x_k)$ we prove a stronger error estimate that 
implies \emph{$l$-step q-linear convergence} for all sufficiently large $l$. 
A sequence that converges $l$-step q-linearly for a single $l$ is r-linearly convergent, but not vice versa. 
We discuss $l$-step q-linear convergence further in Section~\ref{sec_linconv}.
Moreover, we show that $(\nabla f(x_k))$ satisfies a similar error estimate and we assess these theoretical results in the numerical experiments in Section~\ref{sec_experiments}. 
We are not aware of works on \LBFGS~that involve multi-step q-linear convergence, but it frequently appears in results for \BB~methods \cite{DL02,AK20}.
We establish the improved rates for our method in \Cref{thm_linconv} under the aforementioned assumptions.
In addition, we infer that if $f$ is strongly convex in $\Omega$, then the result also holds for classical \LBFGS, cf. \Cref{rem_thmlinconv}. 
This improves existing results such as \cite[Theorem~7.1]{LN89} by lowering its assumptions while sharpening its conclusion.

There are few works that provide a convergence analysis of \LBFGS~in Hilbert space. 
In fact, a Hilbert space setting is only considered in our work~\cite{AMM23} for a structured \LBFGS~method
and in \cite{AK20,AK22} for the \BBB~method.
For BFGS, characterizations of the update and convergence analyses are available in Hilbert space \cite{GriewankInfDimBFGS,Gilbert1989,MayorgaQuintanaInfDimQN,PetraInfDimQN}. 
Still, the techniques that we use differ from those in the literature on \LBFGS~and BFGS. 
For instance, by not relying on traces we avoid the restriction to work with trace class operators, 
which would severely limit the applicability of our results for infinite dimensional~$\CX$ because trace class operators are compact
and require $\CX$ to be separable.  
Studying \LBFGS~in Hilbert space is valuable because when the (globalized) \LBFGS~method is applied to a discretization of an infinite dimensional optimization problem, its numerical performance is usually closely related to the convergence properties of the method on the infinite dimensional level, see for instance the example in \cite[Section~1]{GriewankInfDimBFGS}. 
This relationship appears, among others, in the form of \emph{mesh independence} \cite{ABPR86,KS87,AK22}.
In Section~\ref{sec_experiments} we validate the method's mesh independence numerically for an optimal control problem.		

BFGS and \LBFGS~have been modified in various ways to obtain global convergence for nonconvex objectives. 
For line search based methods we are aware of \emph{cautious updating} \cite{LiFu01,WLQ06,BJRT22,YZZ22,LWH22}, 
\emph{modified updating} \cite{LF01,WLQ06,KD15,YWS20}, 
\emph{dam\-ped updating} \cite{P78,G91,ABGP14,S16} and 
\emph{modification of the line search} \cite{YWL17,HD20}. 
Other options for globalization include \emph{trust region approaches}, cf. \cite{BGZY17,BBEM19,BMPS22} 
and the references therein,
\emph{iterative regularization} \cite{L14,TSY22,KS23} as well as \emph{robust BFGS} \cite{Y22}.
For \BBB~methods global and linear convergence in a nonconvex setting is shown in \cite{DL02}.
As pointed out before, however, there is no \LBFGS-type method available for nonconvex objectives that converges globally in the sense that every cluster point is stationary while also recovering classical \LBFGS. 
For the standard cautious updating from \cite{LF01} we are not aware of works that show that every cluster point is stationary. 

A joint implementation of the new method and standard \LBFGS~is available on \textsc{arXiv} with the \href{https://doi.org/10.48550/arXiv.2401.03805}{preprint of this paper}. 

\subsection{Organization and notation}

The paper is organized as follows. To set the stage, Section~\ref{sec_LBFGS} recalls the classical \LBFGS~method. 
In Section~\ref{sec_LBFGSM} we present the modified \LBFGS~method. 
Its convergence is studied in Section~\ref{sec_convana}. 
Section~\ref{sec_experiments} contains numerical experiments and
Section~\ref{sec_conclusion} concludes with a summary.

\emph{Notation}. We use $\N=\{1,2,3,\ldots\}$ and $\N_0=\N\cup\{0\}$.
In common abuse of notation \cite{GriewankInfDimBFGS,Gr87}
the scalar product of $v,w\in\CX$ is indicated by $v^T w$
and the linear functional $w\mapsto v^T w$ is denoted by $v^T$. 
In particular, due to the Riesz representation theorem 
there is a unique element of $\CX$, denoted $\nabla f(x)$, such that $f'(x)=\nabla f(x)^T$. 
Note that for $\CX=\R^N$ this implies that the operator $^T$ equals transposition only if the Euclidean inner product is used
and similarly the gradient is not generally equal to the vector of partial derivatives. 
We write $M\in\Lin(\CX)$ if $M:\CX\rightarrow\CX$ is 
a bounded linear operator with respect to the operator norm. 
An $M\in\Lin(\CX)$ is self-adjoint iff it satisfies $x^T M y = (Mx)^T y$ for all $x,y\in\CX$.
It is positive definite, respectively, positive semi-definite
iff it is self-adjoint and there exists $\beta>0$, respectively, $\beta\geq 0$ such that 
$x^T M x\geq \beta\norm{x}^2$ for all $x\in\CX\setminus\{0\}$. 
We set $\LinPDX:=\{M\in\Lin(\CX): \,M \text{ is positive definite}\}$.
	

\section{The classical \LBFGS~method}\label{sec_LBFGS}

In this section we discuss the aspects of the classical \LBFGS~method for \cref{eq_prob} that are relevant to this work. The pseudo code of the method is given below as Algorithm~\ref{alg_LBFGS}. 

\begin{algorithm2e}[h!]
	\SetAlgoRefName{LBFGS}
	\DontPrintSemicolon
	\caption{Inverse \LBFGS~method}
	\label{alg_LBFGS}
	\KwIn{$x_0\in \CX$, $m\in\N_0$}
	Let $\CI:=\emptyset$\\
	\For{$k=0,1,2,\ldots$}
	{
		\lIf{$\norm{\nabla f(x_k)}=0$}{\outp~$x_k$ and \brk}
		Choose $H_k^{(0)}\in\LinPDX$\tcp*{Often, $H_k^{(0)}=\gamma_k I$ for some $\gamma_k>0$}
		Compute $d_k:=-H_k\nabla f(x_k)$ from $H_k^{(0)}$ and the stored pairs 
		$\{(s_j,y_j)\}_{j\in\CI}$\label{line_tlrunstrbfgs}\\ 
	Compute step size $\alpha_k>0$\\ 
	Compute $s_k:=\alpha_k d_k$, $x_{k + 1}:=x_k + s_k$, $y_k:=\nabla f(x_{k+1})-\nabla f(x_k)$\\
	\lIf{$y_k^T s_k > 0$}{append $(s_k,y_k)$ to storage and redefine $\CI:=\CI\cup\{k\}$\label{line_skykappendlbfgs}}
	\lIf{$\lvert\CI\rvert>m$}{let $n:=\min\,\CI$, remove $(s_n,y_n)$ from storage and redefine $\CI:=\CI\setminus\{n\}$\label{line_skykremovelbfgs}}
}
\end{algorithm2e}

\subsection{The \LBFGS~operator $H_k$}

We recall the definition of the operator $H_k\in\Lin(\CX)$ appearing in Line~\ref{line_tlrunstrbfgs}. 
It is easy to see that $\lvert\CI\rvert\leq m$ holds in Line~\ref{line_tlrunstrbfgs} for any $k$, so there is $r=r(k)$ with $0\leq r\leq m$ such that $\CI=\{k_0,k_1,\ldots,k_{r-1}\}\subset\N_0$.
If $\CI=\emptyset$, we let $H_k:=H_k^{(0)}$. Otherwise, we simplify notation by identifying $k_j$ with $j$ for $j=0,\ldots,r-1$.
That is, $\CI=\{0,1,\ldots,r-1\}$ and the storage is given by $\{(s_j,y_j)\}_{j=0}^{r-1}$ with $y_j^T s_j>0$ for $j=0,\ldots,r-1$ (since $(s_k,y_k)$ is only stored if $y_k^T s_k>0$). 
Endowed with this notation, the operator $H_k$ is obtained as $H_k:=H_k^{(r)}$ from the \emph{seed matrix} $H_k^{(0)}$ and the storage 
$\{(s_j,y_j)\}_{j\in\CI}$ by means of the recursion
\begin{equation}\label{eq_defLBFGSupd2}
	H_k^{({j+1})} = V_{j}^T H_k^{(j)} V_{j} + \rho_{j} s_{j} s_{j}^T, \qquad j=0,\ldots,r-1,
\end{equation}
where $V_j:=I-\rho_{j} y_{j} s_{j}^T$ and $\rho_{j} := (y_{j}^T s_{j})^{-1}$. 
From the Sherman--Morrison formula we infer that $B_k=H_k^{-1}$	can be obtained by setting $B_k:=B_k^{(r)}$ in the recursion
\begin{equation}\label{eq_defLBFGSupd}
	B_k^{(j+1)}:=B_k^{(j)}-\frac{B_k^{(j)}s_j (B_k^{(j)}s_j)^T}{(B_k^{(j)}s_j)^T s_{j}}+\frac{y_{j}y_{j}^T}{y_{j}^T s_{j}}, \qquad j=0,\ldots,r-1,
\end{equation}
where $B_k^{(0)}:=(H_{k}^{(0)})^{-1}$. 
We stress that for $\CX=\R^N$ the transpose operator in \cref{eq_defLBFGSupd2} and \cref{eq_defLBFGSupd} does \emph{not} agree with the transposition of a vector unless the Euclidean scalar product is used on $\CX$, 
cf. the paragraph \emph{Notation} above. 
In practice we do not form $H_k$, but compute the search direction 
$d_k = -H_k\nabla f(x_{k})$ in a matrix free way through the \emph{two-loop recursion} \cite[Algorithm~7.4]{NW06}
(with the transpositions replaced by scalar products). 
This computation is very efficient, enabling the use of \LBFGS~for large-scale problems. 

\subsection{Choice of the seed matrix $H_k^{(0)}$}

The most common choice of the seed matrix $H_k^{(0)}$ is $H_k^{(0)}=\gamma_k I$, where 
\begin{equation}\label{eq_choceofscfct}
	\gamma_k = \frac{y_{k-1}^T s_{k-1}}{\norm{y_{k-1}}^2},
\end{equation}
cf. \cite[(7.20)]{NW06}, but other choices of $\gamma_k$ and of $H_k^{(0)}$ have also been studied, e.g., in \cite{Oren1982,LN89,Gilbert1989} and \cite{Marjugi2013,An21}.
In particular, the following two values are well-known choices for $\gamma_{k+1}$, where the index shift is for later reference. 	

\begin{definition}\label{def_tau_B}
	Let $k\in\N_0$ and let $(s_k,y_k)\in\CX\times\CX$ with $y_k^T s_k>0$. We set
	\begin{equation*}
		\gamma_{k+1}^\Hy:=\frac{y_k^T s_k}{\norm{y_k}^2} 
		\qquad\text{ and }\qquad
		\gamma_{k+1}^\Hs:=\frac{\norm{s_k}^{2}}{y_k^T s_k}.
	\end{equation*}
\end{definition}

To simplify the presentation of the modified \LBFGS~method in Section~\ref{sec_LBFGSM}, we restrict attention to choices of the form 
$H_k^{(0)}=\gamma_k I$ with $\gamma_k\in[\gamma_{k}^\Hy,\gamma_{k}^\Hs]$.
For memory size $m=0$ this yields $d_k=-\gamma_k\nabla f(x_k)$,
which shows that in this case Algorithm~\ref{alg_LBFGS} may be viewed as a globalized \BBB~method. For \BBB~methods, the values $\gamma_{k}^\Hy$ and $\gamma_{k}^\Hs$ were introduced in \cite{BB88} and they are sometimes referred to as the two \BB~step sizes. 
The restriction $\gamma_k\in[\gamma_{k}^\Hy,\gamma_{k}^\Hs]$ appears for instance in the spectral gradient methods of \cite{DHL19}, a work that provides ample references to related \BBB~methods.

\begin{remark}
	The Cauchy--Schwarz inequality implies $\gamma_{k+1}^\Hy\leq\gamma_{k+1}^\Hs$.
	Furthermore, it is well-known that for twice differentiable $f$ with positive definite Hessians, the numbers $\gamma_{k+1}^\Hy$ and $\gamma_{k+1}^\Hs$ fall within the spectrum of the inverse of the \emph{averaged Hessian} $\int_0^1\nabla^2f(x_k+t s_k)\,\mathrm{d}t$,  cf. \cite{AMM23} for a statement in Hilbert space. 
	If $f$ is quadratic with positive definite Hessian $B$, we have 
	$B s_k=y_k$, hence $\gamma_{k+1}^\Hs$ is the inverse of a Rayleigh quotient of $B$ and $\gamma_{k+1}^\Hy$ is a Rayleigh quotient of $B^{-1}$. 
\end{remark}

\subsection{Choice of the step size $\alpha_k$}
Most often, Algorithm~\ref{alg_LBFGS} is employed with a line search that ensures the satisfaction of the Wolfe--Powell conditions or the strong Wolfe--Powell conditions.
That is, for constants $\sigma\in(0,1)$ and $\eta\in (\sigma,1)$, the step sizes $\alpha_k$ are selected in such a way that they 
satisfy the Armijo condition 
\begin{equation}\label{eq_armijocond}
	f(x_{k+1}) \leq f(x_k) + \alpha_k\sigma\nabla f(x_k)^T d_k
\end{equation}
and either the curvature condition or the strong curvature condition	
\begin{equation}\label{eq_WolfePowellcond}
	\nabla f(x_{k+1})^T d_k \geq \eta\nabla f(x_k)^T d_k, \enspace\text{resp. }\enspace
	\left\lvert\nabla f(x_{k+1})^T d_k\right\rvert \leq \eta\left\lvert\nabla f(x_k)^T d_k\right\rvert.
\end{equation}
This entails that $y_k^T s_k>0$, so the current secant pair $(s_k,y_k)$ is guaranteed to enter the storage in Line~\ref{line_skykappendlbfgs},
which slightly simplifies the algorithm, e.g., \cite[Algorithm~7.5]{NW06}, 
and ensures that the storage contains the most recent secant information. 
Since we also consider line searches that only involve the Armijo condition \cref{eq_armijocond}, we may have $y_k^T s_k\leq 0$. 
In this case the pair $(s_k,y_k)$ is not stored because using it in the construction of $H_{k+1}$ would	result in $H_{k+1}$ not being positive definite. 
While this skipping of secant information seems undesirable, it can still pay off to work only with the Armijo condition 
\cite{ACG18,AMM23}. 
If step sizes are based on \cref{eq_armijocond} only, we demand that they are computed by backtracking. 
That is, if $\alpha_{k,i}$ does not satisfy \cref{eq_armijocond}, then the next trial step size $\alpha_{k,i+1}$ is chosen from $[\beta_1\alpha_{k,i},\beta_2\alpha_{k,i}]$, where $i=0,1,\ldots$, $\alpha_{k,0}:=1$ and $0<\beta_1\leq\beta_2<1$ are constants. 
This includes step size strategies that involve interpolation \cite[Section~6.3.2]{DS96}.
	
\subsection{Effect of update skipping on the storage}\label{sec_removpairsfromstoragenonWP}

Observe that if no new pair enters the storage in Line~\ref{line_skykappendlbfgs}, then no old pair will be removed in Line~\ref{line_skykremovelbfgs}, so 
we allow older information to be retained. 
This is a design choice and other variants are conceivable, e.g., 
we could always remove $(s_{k-m},y_{k-m})$ from the storage and set $\CI:=\CI\setminus\{k-m\}$ in Line~\ref{line_skykremovelbfgs}. 
This would not affect the convergence analysis of this work in a meaningful way. 


\section{The modified \LBFGS~method}\label{sec_LBFGSM}	

In this section we present the new method, point out its novelties and discuss how they
relate to \emph{cautious updating} \cite{LiFu01} introduced by Li and Fukushima. 

To state the pseudo code of the method, let us denote 
\begin{equation*}
	q:\CX\times\CX\rightarrow\R, \qquad q(s,y):=
	\begin{cases}
		\min\left\{\frac{y^T s}{\norm{s}^2},\,\frac{y^T s}{\norm{y}^2}\right\} & \text{ if }s\neq 0 \text{ and } y\neq 0,\\
		0 & \text{ else}.
	\end{cases}
\end{equation*}
The new method Algorithm~\ref{alg_hybrid} (=\,\underline{\LBFGS}, \underline{m}odified) reads as follows. 

\begin{algorithm2e}[ht!]
	\SetAlgoRefName{LBFGSM}
	\DontPrintSemicolon
	\caption{The modified inverse \LBFGS~method}
	\label{alg_hybrid}
	\KwIn{$x_0\in \CX$, $m\in\N_0$, $c_0\in(0,1]$, $c_1,c_2\in(0,\infty)$}
Let $\CI:=\emptyset$, $\gamma_0^\Hy:=0$, $\gamma_0^\Hs:=\infty$\\
\For{$k=0,1,2,\ldots$}
{
\lIf{$\norm{\nabla f(x_k)}=0$}{\outp~$x_k$ and \brk\label{line_term}}
Let $\omega_k:=\min\{c_0,c_1\norm{\nabla f(x_k)}^{c_2}\}$\label{line_defomegak}\\
Choose $\gamma_k\in [\gamma_k^\Hy,\gamma_k^\Hs]\cap [\omega_k,\omega_k^{-1}]$ if possible; else choose $\gamma_k\in [\omega_k,\omega_k^{-1}]$
\label{line_choicetauk}\\
Let $H_k^{(0)}:=\gamma_k I$\\
Compute $d_k:=-H_k\nabla f(x_k)$ from $H_k^{(0)}$ and the pairs in $\{(s_j,y_j)\}_{j\in\CI}$ with $q_j\geq\omega_k$\label{line_tlr}\\
Compute step size $\alpha_k>0$\\ 
Compute $s_k:=\alpha_k d_k$, $x_{k+1}:=x_k + s_k$, $y_k:=\nabla f(x_{k+1})-\nabla f(x_k)$\\
\uIf{$y_k^T s_k > 0$ \label{line_acceptanceofysforstorage}}{append $(s_k,y_k)$ to storage, redefine $\CI:=\CI\cup\{k\}$ and let $q_k:=q(s_k,y_k)$\label{line_skykappend}\\
	compute $\gamma_{k+1}^\Hy$ and $\gamma_{k+1}^\Hs$ according to Definition~\ref{def_tau_B}}			
\Else{let $\gamma_{k+1}^\Hy:=0$ and $\gamma_{k+1}^\Hs:=\infty$}
\lIf{$\lvert\CI\rvert>m$}{let $n:=\min\,\CI$, remove $(s_n,y_n)$ from storage and redefine $\CI:=\CI\setminus\{n\}$\label{line_skykremove}}
}
\end{algorithm2e}

The interval $[\gamma_k^\Hy,\gamma_k^\Hs]\cap [\omega_k,\omega_k^{-1}]$ in Line~\ref{line_choicetauk} is well-defined, but it may be empty. In that case,  $[\omega_k,\omega_k^{-1}]$ is used. The latter is nonempty due to $c_0\in(0,1]$. 

\subsection{The two main novelties of Algorithm~\ref{alg_hybrid}}

While Algorithm~\ref{alg_hybrid} still stores any pair $(s_k,y_k)$ with $y_k^T s_k>0$, cf. Line~\ref{line_skykappend}, it does not necessarily use all stored pairs for constructing $H_k$. Instead, in every iteration the updating involves only those of the stored pairs $(s_j,y_j)$, $j\in \CI$, that satisfy $q_j\geq\omega_k$, cf. Line~\ref{line_tlr}, where $q_j=q(s_j,y_j)$.  
We recall that the construction of $H_k$ is described in \cref{eq_defLBFGSupd2}. 
Note that the condition $q_j\geq\omega_k$
relates all stored pairs to the current value $\omega_k$, so a pair $(s_j,y_j)$ in the storage may be used to form $H_k$ in some iterations, but may be skipped in others, until it is removed from the storage. In contrast, in \LBFGS~a pair within the storage is consistently used to form $H_k$ until it is removed from the storage. 
The second novelty concerns the choice of $\gamma_k$. In classical \LBFGS~with Wolfe--Powell step sizes 
we have $\gamma_k^\Hy>0$ for all~$k$, so often $\gamma_k=\gamma_k^\Hy$ is chosen for all~$k$. 
By additionally requiring $\gamma_k\in[\omega_k,\omega_k^{-1}]$, Algorithm~\ref{alg_hybrid} 
further restricts the choice of $\gamma_k$. In particular, this may prevent the choice 
$\gamma_k=\gamma_k^\Hy$ for some $k$. 
The key point is that these two modifications allow us to bound 
$\norm{H_k}$ and $\norm{H_k^{-1}}$ in terms of $\omega_k$, cf. \Cref{lem_BkandBkinversebounded},
which is crucial for establishing the global convergence of \ref{alg_hybrid} in Theorems~\ref{thm_globconv} and \ref{thm_clusterpointsarestationar}
(note that $\omega_k$ is related to $\norm{\nabla f(x_k)}$).
We are not aware of other works on cautious updating that have achieved global convergence in the sense of \Cref{thm_globconv} or \ref{thm_clusterpointsarestationar} without the use of fixed bounds for $\norm{H_k}$ and $\norm{H_k^{-1}}$. Such bounds can severely degrade the performance, cf. also the comment at the end of Section~\ref{sec_relcu}. 

\subsection{Relation of Algorithm~\ref{alg_hybrid} to cautious updating}\label{sec_relcu}

We emphasize that the two main novelties of Algorithm~\ref{alg_hybrid} are based on the \emph{cautious updating} introduced by Li and Fukushima in \cite{LiFu01} for the BFGS method. There, the BFGS update based on $(s_k,y_k)$ is applied if $y_k^T s_k/\norm{s_k}^2\geq \epsilon\norm{\nabla f(x_k)}^\alpha$
for positive constants $\epsilon$ and $\alpha$, otherwise the update is skipped; cf. \cite[(2.10)]{LiFu01}.
This condition for skipping the update has been employed many times in the literature, both for BFGS and for \LBFGS. However, 
to the best of our knowledge, in the existing variants once a pair $(s_j,y_j)$ is used for updating, it is also used for updating in every subsequent iteration (until it is removed from the storage in case of \LBFGS). Clearly, this is less flexible than deciding in each iteration which of the stored pairs are used for updating. Moreover, in classical cautious updating the decision whether to store a pair $(s_j,y_j)$ (and thus, to use it for updating) is based on $\norm{\nabla f(x_j)}$,
whereas in Algorithm~\ref{alg_hybrid} the decision to involve a stored pair $(s_j,y_j)$ in the updating is based on the most recent norm $\norm{\nabla f(x_k)}$ rather than older norms. 
Finally, involving $\norm{\nabla f(x_k)}$ in the choice of $H_k^{(0)}$ seems to be new altogether apart from our very recent work \cite{AMM23}.
As it turns out, the novel form of cautious updating used in Algorithm~\ref{alg_hybrid} 
allows to bound $\norm{H_k}$ and $\norm{H_k^{-1}}$ in terms of $\norm{\nabla f(x_k)}^{-1}$, which 
is crucial for establishing global and linear convergence. 
We emphasize that the bounds for $\norm{H_k}$ and $\norm{H_k^{-1}}$ tend to infinity when $\norm{\nabla f(x_k)}$ tends to zero,
cf. \Cref{lem_BkandBkinversebounded}. 
This is highly desirable because it indicates that $\norm{H_k}$ and $\norm{H_k^{-1}}$ can become as large as necessary for $k\to\infty$, whereas any preset lower bound for $\gamma_k$ or $q(s_k,y_k)$ uniformly bounds the spectrum of $H_k$ and 
may therefore limit the ability of $H_k$ to capture the spectrum of the inverse Hessian, which would slow down the convergence.

\subsection{A further novelty of Algorithm~\ref{alg_hybrid}: $\nabla f$ not (globally) Lipschitz}

Another novelty of Algorithm~\ref{alg_hybrid} concerns the assumption of Lipschitz continuity of $\nabla f$ in the level set associated to $x_0$.
Specifically, if $\nabla f$ satisfies this assumption, then $q$ can be replaced by the simpler function 
$\hat q(s,y):=y^T s/\norm{s}^2$ for $s,y\neq 0$, $\hat q(s,y)=0$ else, without meaningfully affecting the convergence results of this paper. 
The decision criterion $\hat q(s_j,y_j)\geq\omega_k$ in Line~\ref{line_tlr} of Algorithm~\ref{alg_hybrid} then resembles the original cautious updating more clearly. 
However, using $q$ instead of $\hat q$ enables global convergence if $\nabla f$ is continuous and linear convergence if $\nabla f$ is Lipschitz continuous near cluster points, cf. \Cref{thm_clusterpointsarestationar} and \Cref{thm_linconv}.
That is, using $q$ instead of $\hat q$ allows us to eliminate, respectively, weaken substantially the assumption of $\nabla f$ being Lipschitz in the entire level set. We are not aware of other works on cautious updating that have achieved this. 

\subsection{A slightly larger interval for the choice of $\gamma_k$}

Without meaningfully affecting the convergence theory of this paper we could 
replace the interval $[\gamma_k^\Hy,\gamma_k^\Hs]$ in Line~\ref{line_choicetauk} of Algorithm~\ref{alg_hybrid} by
a larger interval $[\gamma_k^{\min},\gamma_k^{\max}]$, where
$\gamma_k^{\min}:=\min\{c_3,c_4\gamma_k^\Hy\}$ and $\gamma_k^{\max}:=\max\{C_3,C_4\gamma_k^\Hs\}$ for positive constants $c_3,c_4,C_3,C_4$ with $c_3\leq C_3$.
This allows, in particular, to include the choice $\gamma_k=1$ for all $k$, i.e., $H_k^{(0)}=I$,
that also appears in the literature. However, in our assessment $H_k^{(0)}=\gamma_k I$ with $\gamma_k$ according to \cref{eq_choceofscfct} is used much more often, so to keep the notation light we work with $[\gamma_k^\Hy,\gamma_k^\Hs]$. 
(If $m=0$ and $\gamma_k=1$ for all $k$, \ref{alg_hybrid} is the method of steepest descent.) 


\section{Convergence analysis}\label{sec_convana}

This section is devoted to the convergence properties of Algorithm~\ref{alg_hybrid}. 
We start with its well-definedness in Section~\ref{sec_welldef}, after which 
we focus on the global and local convergence in Sections~\ref{sec_globconv} and \ref{sec_linconv}. 

\subsection{Well-definedness}\label{sec_welldef}

We recall some estimates to infer that Algorithm~\ref{alg_hybrid} is well-defined.

\begin{lemma}[cf. {\cite[Lemma~4.1]{AMM23}}]\label{lem_generalresultonnormbounds}
Let $M\in\N_0$ and $\kappa_1,\kappa_2>0$. For all $j\in\{0,\ldots,M-1\}$ let $(s_j,y_j)\in\CX\times\CX$
satisfy 
\begin{equation*}
	\frac{y_j^T s_j}{\norm{s_j}^2}\geq \frac{1}{\kappa_1}\qquad\text{ and }\qquad
	\frac{y_j^T s_j}{\norm{y_j}^2}\geq \frac{1}{\kappa_2}. 
\end{equation*}
Moreover, let $H^{(0)}\in\LinPDX$. Then its \LBFGS~update $H$, obtained from the recursion \cref{eq_defLBFGSupd2} using $H_k^{(0)}=H^{(0)}$ and $r=M$, satisfies $H\in\LinPDX$ and
\begin{equation}\label{eq_estlargestEV}
	\norm{H^{-1}}\leq\norm{(H^{(0)})^{-1}} + M \kappa_2
\end{equation}
as well as
\begin{equation}\label{eq_estsmallestEV}
	\norm{H}\leq 5^M\max\bigl\{1,\norm{H^{(0)}}\bigr\}\max\bigl\{1,\kappa_1^M,(\kappa_1 \kappa_2)^M\bigr\}.
\end{equation}
\end{lemma}	

The following assumption ensures that Algorithm~\ref{alg_hybrid} is well-defined.

\begin{assumption}\label{ass_basic}
\phantom{to enforce linebreak}
\begin{enumerate}
	\item[1)] The function $f:\CX\rightarrow\R$ is continuously differentiable and bounded below. 
	\item[2)] The step size $\alpha_k$ in \ref{alg_hybrid} either satisfies the Armijo condition \cref{eq_armijocond}
	and is computed by backtracking for all $k$, or it satisfies the Wolfe--Powell conditions for all $k$, cf. \cref{eq_WolfePowellcond}. 
\end{enumerate}
\end{assumption}

\begin{lemma}\label{lem_welldef}
Let \Cref{ass_basic} hold. Then Algorithm~\ref{alg_hybrid} 
either terminates in Line~\ref{line_term} with an $x_k$ satisfying $\nabla f(x_k)=0$
or it generates a sequence $(f(x_k))$ that is strictly monotonically decreasing and convergent. 
\end{lemma}		

\begin{proof}
After observing that $[\omega_k,\omega_k^{-1}]\neq\emptyset$ for all $k$ due to $c_0\leq 1$,
the proof is similar to that of \cite[Lemma~4.5]{AMM23}. 
\end{proof}	

Applying \Cref{lem_generalresultonnormbounds} to Algorithm~\ref{alg_hybrid} yields the following result.

\begin{lemma}\label{lem_BkandBkinversebounded}
Let \Cref{ass_basic} hold and let $(x_k)$ be generated by Algorithm~\ref{alg_hybrid}. 
Then 
\begin{equation*}
	\norm{H_k^{-1}}\leq \left(m+1\right) \omega_k^{-1}
	\qquad\text{ and }\qquad
	\norm{H_k}\leq 5^m\max\Bigl\{1,\omega_k^{-(2m+1)}\Bigr\}
\end{equation*}
are satisfied for all $k\in\N_0$, where $m$ is the constant from Algorithm~\ref{alg_hybrid}.
\end{lemma}	

\begin{proof}
The acceptance criterion $q(s_j,y_j)\geq\omega_k$ in Line~\ref{line_tlr} of Algorithm~\ref{alg_hybrid} 
implies that $y_j^T s_j/\norm{s_j}^2\geq\omega_k$ and $y_j^T s_j/\norm{y_j}^2\geq\omega_k$ for all $(s_j,y_j)$ that are involved in forming $H_k$.
Moreover, Line~\ref{line_choicetauk} ensures that $\omega_k\leq \gamma_k\leq\omega_k^{-1}$. 
Using \cref{eq_estlargestEV}, respectively, \cref{eq_estsmallestEV} we thus deduce that 
\begin{equation*}
	\norm{H_k^{-1}}\leq\norm{(H_k^{(0)})^{-1}} + m\omega_k^{-1}
	= \norm{I}\gamma_k^{-1} + m\omega_k^{-1}
	\leq \omega_k^{-1} + m\omega_k^{-1}
\end{equation*}
and 
\begin{equation*}
	\norm{H_k}\leq 5^m\max\bigl\{1,\omega_k^{-1}\bigr\}\max\bigl\{1,\omega_k^{-2m}\bigr\}
	\leq 5^m\max\bigl\{1,\omega_k^{-(2m+1)}\bigr\}, 
\end{equation*}
establishing the assertions. 
\end{proof}	

The previous lemma implies useful bounds. 

\begin{corollary}\label{cor_BkandBkinversebounded}
Let \Cref{ass_basic} hold and let $(x_k)$ be generated by Algorithm~\ref{alg_hybrid}. 
Then there are constants $c,C>0$ such that
\begin{equation}\label{eq_boundonlengthofgradientsandsteps}
	c\min\Bigl\{1,\norm{\nabla f(x_k)}^{c_2}\Bigr\}\leq \frac{\norm{d_k}}{\norm{\nabla f(x_k)}}\leq C\max\Bigl\{1,\norm{\nabla f(x_k)}^{-(2m+1)c_2}\Bigr\}
\end{equation}
as well as 
\begin{equation*}
	\frac{\lvert\nabla f(x_k)^T d_k\rvert}{\norm{d_k}} \geq c\norm{\nabla f(x_k)}\min\Bigl\{1,\norm{\nabla f(x_k)}^{2(m+1)c_2}\Bigr\}.
\end{equation*}
are satisfied for all $k\in\N_0$, where $c_2$ and $m$ are the constants from Algorithm~\ref{alg_hybrid}.
\end{corollary}

\begin{proof}
Since $d_k=-H_k\nabla f(x_k)$, the estimates \cref{eq_boundonlengthofgradientsandsteps} readily follow from those for $\norm{H_k}$ and $\norm{H_k^{-1}}$.
Concerning the final inequality we have
\begin{equation*}
	\begin{split}
		-\nabla f(x_k)^T d_k = d_k^T H_k^{-1} d_k 
		&\geq \norm{d_k}^2\norm{H_k}^{-1}\\
		&\geq c\norm{d_k}\norm{\nabla f(x_k)}\min\Bigl\{1,\norm{\nabla f(x_k)}^{(2m+2)c_2}\Bigr\},
	\end{split}
\end{equation*}
where we used the first inequality of \cref{eq_boundonlengthofgradientsandsteps} and the estimate for $\norm{H_k}$. 
\end{proof}	

\subsection{Global convergence}\label{sec_globconv}

In this section we establish the global convergence of Algorithm~\ref{alg_hybrid}. In fact, we prove several types of global convergence under different assumptions.

The level set for Algorithm~\ref{alg_hybrid}, respectively, the level set with $\delta$ vicinity are given by
\begin{equation*}
\Omega:=\Bigl\{x\in\CX: \; f(x)\leq f(x_0)\Bigr\}\qquad\text{ and }\qquad \Omega_\delta:=\Omega+\Ballop{\delta}(0).
\end{equation*}
The first result relies on the following assumption.

\begin{assumption}\label{ass_globconv}	
\phantom{to create linebreak}
\begin{enumerate}
	\item[1)] \Cref{ass_basic} holds.
	\item[2)] The gradient of $f$ is uniformly continuous in $\Omega$, i.e., for all sequences\linebreak $(x_k)\subset\Omega$ and $(s_k)\subset\CX$ satisfying $(x_k+s_k)\subset\Omega$ and $\lim_k s_k=0$, there holds $\lim_k\,\norm{\nabla f(x_k+s_k)-\nabla f(x_k)}=0$. 
	\item[3)] If Algorithm~\ref{alg_hybrid} uses Armijo with backtracking, then 
	there is $\delta>0$ such that $f$ or $\nabla f$ is uniformly continuous in $\Omega_\delta$.	
\end{enumerate}		
\end{assumption}

\begin{remark}
If $\CX$ is finite dimensional and $\Omega$ is bounded, part~2) can be dropped since it follows from the continuity of $\nabla f$;
similarly for part~3). 
\end{remark}

To prove global convergence for step sizes that satisfy the \WP~conditions, 
we will use the following less stringent form of the \emph{Zoutendijk condition} from \cite{A22}. 
In contrast to the original Zoutendijk condition, cf. e.g., \cite[Theorem~3.2]{NW06}, it holds without Lipschitz continuity of the gradient.

\begin{lemma}{(cf. \cite[Theorem~2.28]{A22})}\label{lem_fundamentalstepsizelemma}
Let \Cref{ass_globconv} hold and let $(x_k)$ be generated by Algorithm~\ref{alg_hybrid} with the Wolfe--Powell conditions.
Then $\lim_{k\to\infty}\frac{\lvert\nabla f(x_k)^T d_k\rvert}{\norm{d_k}} = 0$.
\end{lemma}
		
\begin{remark}\label{rem_sWPwithoutuniformcontinuity}
By following the proof of \cite[Theorem~2.28]{A22} it is not difficult to show 
that if we replace \Cref{ass_globconv} by \Cref{ass_basic} (i.e., drop the uniform continuity of $\nabla f$)
and additionally assume that $(x_k)_K$ converges, then 
$\lim_{K\ni k\to\infty}\frac{\lvert\nabla f(x_k)^T d_k\rvert}{\norm{d_k}} = 0$.
\end{remark}

As a main result we establish that Algorithm~\ref{alg_hybrid} is globally convergent under \Cref{ass_globconv}. 

\begin{theorem}\label{thm_globconv}
Let \Cref{ass_globconv} hold. Then
Algorithm~\ref{alg_hybrid} either terminates after finitely many iterations with an $x_k$ that satisfies $\nabla f(x_k)=0$ or it generates a sequence $(x_k)$ with
\begin{equation}\label{eq_liminfiszero}
	\lim_{k\to\infty}\,\norm{\nabla f(x_k)} = 0.
\end{equation}
In particular, every cluster point of $(x_k)$ is stationary.
\end{theorem}

\begin{proof}
The case that Algorithm~\ref{alg_hybrid} terminates is clear,
so let us assume that it generates a sequence $(x_k)$.
If $(x_k)$ satisfies \cref{eq_liminfiszero}, then by continuity it follows that every cluster point $\xopt$ of $(x_k)$ satisfies $\nabla f(\xopt)=0$. Hence, it only remains to establish \cref{eq_liminfiszero}. 
We argue by contradiction, so suppose that \cref{eq_liminfiszero} were false. Then there is a subsequence $(x_k)_K$ of $(x_k)$ and an
$\epsilon>0$ such that 
\begin{equation}\label{eq_contrassgradawayfromzero}
	\norm{\nabla f(x_k)}\geq\epsilon \qquad\forall k\in K.
\end{equation}
Since $\alpha_k$ satisfies the Armijo condition for all $k\geq 0$, we infer from $-\nabla f(x_k)^T d_k = d_k^T H_k^{-1} d_k$ that
\begin{equation*}
	\begin{split}
		\sigma\sum_{k\in K}\alpha_k\frac{\norm{d_k}^2}{\norm{H_k}}
		&\leq -\sigma\sum_{k\in K}\alpha_k\nabla f(x_k)^T d_k\\
		&\leq \sum_{k\in K} \left[f(x_k)-f(x_{k+1})\right] 
		\leq \sum_{k=0}^\infty \left[f(x_k)-f(x_{k+1})\right] 
		<\infty,
	\end{split}
\end{equation*}
where we used that $(f(x_k))$ is monotonically decreasing by \Cref{lem_welldef}. 
From \Cref{lem_BkandBkinversebounded} and \cref{eq_contrassgradawayfromzero} we obtain 
$\sup_{k\in K}\norm{H_k}<\infty$, which yields that the sum $\sum_{k\in K}\alpha_k\norm{d_k}^2$ is finite.
From the first inequality in \cref{eq_boundonlengthofgradientsandsteps} and \cref{eq_contrassgradawayfromzero} we infer that 
\begin{equation}\label{eq_dkboundedawayfromzero}
	\exists c>0: \quad \norm{d_k}\geq c \quad\forall k\in K.
\end{equation}
Hence, $\sum_{k\in K}\alpha_k\norm{d_k}^2<\infty$ implies 
$\sum_{k\in K}\alpha_k\norm{d_k}<\infty$ and $\sum_{k\in K}\alpha_k<\infty$, 
thus
\begin{equation}\label{eq_yal}
	\lim_{K\ni k\to\infty}\alpha_k = \lim_{K\ni k\to\infty}s_k = 0.
\end{equation}
We now distinguish two cases depending on how the step sizes are determined.\\
\textbf{Case 1: The step sizes are computed by Armijo with backtracking}\\		
From \cref{eq_yal} and the construction of the backtracking we infer that for all sufficiently large $k\in K$ there is $\beta_k\in (0,\beta_1^{-1}]$
such that \cref{eq_armijocond} is violated for $\hat\alpha_k:=\beta_k\alpha_k$. Therefore, 		
\begin{equation*}
	-\hat\alpha_k\sigma\nabla f(x_k)^T d_k > f(x_k)-f(x_k+\hat\alpha_k d_k)
	= -\hat\alpha_k\nabla f(x_k+\theta_k\hat\alpha_kd_k)^T d_k
\end{equation*}
for all these $k$ and $\theta_k\in(0,1)$. Multiplying by $-\hat\alpha_k^{-1}$ and subtracting $\nabla f(x_k)^T d_k$ we obtain 
\begin{equation*}
	(\sigma-1)\nabla f(x_k)^T d_k < \left[\nabla f(x_k+\theta_k\hat\alpha_kd_k) - \nabla f(x_k)\right]^T d_k.
\end{equation*}
Due to $\sigma<1$ and $-\nabla f(x_k)^T d_k = d_k^T H_k^{-1} d_k$ there holds for all $k\in K$ sufficiently large 
\begin{equation*}
	(1-\sigma)\frac{\norm{d_k}^2}{\norm{H_k}} < 
	\norm{\nabla f(x_k+\theta_k\hat\alpha_kd_k)-\nabla f(x_k)}\norm{d_k}.
\end{equation*}
Because of $h:=\sup_{k\in K}\norm{H_k}<\infty$ this entails for all large $k\in K$ that
\begin{equation}\label{eq_qssi}
	(1-\sigma)\norm{d_k}< h\norm{\nabla f(x_k+\theta_k\hat\alpha_kd_k)-\nabla f(x_k)}.
\end{equation}
As $\theta_k\in(0,1)$ and $\beta_k\leq\beta_1^{-1}$ for all $k\in K$, it follows from \cref{eq_yal} that $\theta_k\hat\alpha_k d_k=\theta_k\beta_k s_k\to 0$ for $K\ni k\to\infty$. 
If $\nabla f$ is uniformly continuous in $\Omega_\delta$, then \cref{eq_qssi}
implies $\lim_{K\ni k\to\infty}\norm{d_k}=0$, which contradicts \cref{eq_dkboundedawayfromzero}.
If $f$ is uniformly continuous in $\Omega_\delta$, then 
$x_k+\theta_k\hat\alpha_k\norm{d_k}\in\Omega$ for all large $k\in K$,
so the uniform continuity of $\nabla f$ in $\Omega$ yields $\lim_{K\ni k\to\infty}\norm{d_k}=0$ in \cref{eq_qssi}, contradicting \cref{eq_dkboundedawayfromzero}.\\		
\textbf{Case 2: The step sizes satisfy the Wolfe--Powell conditions}\\
Combining \Cref{lem_fundamentalstepsizelemma} and \Cref{cor_BkandBkinversebounded} yields 
$\lim_{K\ni k\to\infty}\norm{\nabla f(x_k)}=0$, contradicting \cref{eq_contrassgradawayfromzero}.
\end{proof}	

Cluster points are already stationary under \Cref{ass_basic}.

\begin{theorem}\label{thm_clusterpointsarestationar}
	Let \Cref{ass_basic} hold and let $(x_k)$ be generated by Algorithm~\ref{alg_hybrid}. Then every cluster point of $(x_k)$ is stationary.
\end{theorem}

\begin{proof}
The proof is almost identical to that of \Cref{thm_globconv}. The main difference is that since the subsequence $(x_k)_K$ now converges to some $\xopt$,
continuity of $f$ at $\xopt$ implies that there is $\delta'>0$ such that $\Ballop{\delta'}(\xopt)\subset\Omega$ and for all sufficiently large $k\in K$ there holds $x_k\in\Ballop{\delta'}(\xopt)$. 
Therefore, local assumptions in the vicinity of the cluster point suffice (e.g., continuity instead of uniform continuity). 
In Case~2 of the proof, in particular, we use \Cref{rem_sWPwithoutuniformcontinuity} and \cref{eq_yal} instead of \Cref{lem_fundamentalstepsizelemma}.
\end{proof}

\begin{remark}
\Cref{thm_clusterpointsarestationar} shows global convergence using only continuity of the gradient and no boundedness of the level set.
This result can be extended to essentially any descent method, so it may be of interest beyond this work. 
Surprisingly, we have not found it elsewhere for the \WP~conditions. 
\end{remark}	

If $\CX$ is finite dimensional and $(x_k)$ is bounded, we obtain the conclusions of \Cref{thm_globconv} under the weaker assumptions of \Cref{thm_clusterpointsarestationar}.

\begin{corollary}\label{cor_unifcontfindim}
Let \Cref{ass_basic} hold and let $(x_k)$ be generated by Algorithm~\ref{alg_hybrid}. 
Let $(x_k)$ be bounded and let $\CX$ be finite dimensional. 
Then $(x_k)$ has at least one cluster point and $\lim_{k\to\infty}\norm{\nabla f(x_k)}=0$. 
\end{corollary}	

\begin{proof}
The claims readily follow from \Cref{thm_clusterpointsarestationar}. 
\end{proof}

In the infinite dimensional case we also want \emph{weak} cluster points to be stationary. 
Recall from \Cref{thm_globconv} that $\lim_{k\to\infty}\norm{\nabla f(x_k)}=0$ holds under 
\Cref{ass_globconv}.

\begin{lemma}\label{lem_weakclparestationary}
Let \Cref{ass_basic} hold and let $\nabla f$ be weak-to-weak continuous. 
Let $(x_k)$ be generated by Algorithm~\ref{alg_hybrid} and let $\lim_{k\to\infty}\norm{\nabla f(x_k)}=0$. 
Then every weak cluster point of $(x_k)$ is stationary. 
\end{lemma}	

\begin{proof}
As $\norm{\cdot}$ is weakly lower semicontinuous, the proof is straightforward. 
\end{proof}

\subsection{Linear convergence}\label{sec_linconv}

In this section we prove that Algorithm~\ref{alg_hybrid} converges linearly under mild assumptions and 
that it turns into classical~\LBFGS~under first order sufficient optimality conditions. 
We divide the section into three parts.

\subsubsection{Preliminaries}

The main result on linear convergence, \Cref{thm_linconv}, will show that $(f(x_k))$ converges q-linearly while 
$(x_k)$ and $(\nabla f(x_k))$ satisfy estimates that imply \emph{$l$-step q-linear convergence} for all sufficiently large $l$. 

\begin{definition}
We call $(x_k)\subset\CX$ \emph{$l$-step q-linearly convergent} for some $l\in\N$, iff there exist $\xopt\in\CX$, $\bar k\geq 0$ and $\kappa\in(0,1)$ such that $\norm{x_{k+l}-\xopt}\leq\kappa\norm{x_k-\xopt}$ is satisfied for all $k\geq\bar k$.
\end{definition}

For $l=1$ this is q-linear convergence.	It is easy to see that $l$-step q-linear convergence for an arbitrary $l$ implies r-linear convergence whereas the opposite is not necessarily true. 

We are not aware of works on \LBFGS-type methods that use the concept of $l$-step q-linear convergence. 
However, for \BB-type methods the notion appears in \cite{DL02,AK20,AK22} in convex and nonconvex settings.

We use the following assumption to obtain linear convergence of \ref{alg_hybrid}.

\begin{assumption}\label{ass_linconv}	
\phantom{to create linebreak}
\begin{enumerate}
	\item[1)] \Cref{ass_basic} holds.
	\item[2)] The constant $c_2$ in Algorithm~\ref{alg_hybrid} satisfies $c_2<(2m+1)^{-1}$ if the Armijo rule with backtracking is used, and 
	$c_2<(2m+2)^{-1}$ else.
\end{enumerate}		
\end{assumption}	

Note that the requirement on $c_2$ in part~2) restricts the rate with which $\omega_k$ can go to zero, cf. the definition of $\omega_k$ in Line~\ref{line_defomegak} of Algorithm~\ref{alg_hybrid}.
Due to \Cref{lem_BkandBkinversebounded} this limits the growth rate of the condition number of the \LBFGS~operator $H_k$. 
The assumption on $c_2$ is crucial to show that if a stationary cluster point $\xopt$ belongs to a ball in which $f$ is strongly convex, then $\xopt$ is \emph{attractive}, i.e., the entire sequence $(x_k)$ converges to $\xopt$. This property is the fundamental building block for all subsequent results, including for the transition to classical \LBFGS~in \Cref{thm_lbfgsmislbfgs} and for the linear convergence in \Cref{thm_linconv}. 
In other words, the calibration of the cautious updates according to 2), which we have not seen elsewhere, is of vital importance from a theoretical point of view. 

\begin{lemma}\label{lem_aux}
Let \Cref{ass_linconv} hold and let $(x_k)$ be generated by Algorithm~\ref{alg_hybrid}. 
Let $(x_k)$ have a cluster point $\xopt$ with a convex neighborhood in which $f$ is strongly convex. 
Then $\lim_{k\to\infty}\norm{x_k - \xopt}=0$. 
\end{lemma}	

\begin{proof}
Let $\nbhd$ denote the neighborhood of $\xopt$. The proof consists of two parts.\\
\textbf{Part 1: Cluster points induce vanishing steps.}\\		
First we show that for any $\epsilon>0$ there exists $\delta'>0$ such that for all $k\in\N_0$ the implication 
$\norm{x_k-\xopt}<\delta'\Rightarrow\norm{s_k}<\epsilon$ holds true. 
Apparently, to prove this it suffices to consider $\epsilon$ so small that $\Ballop{\epsilon}(\xopt)\subset\nbhd$. 
Let such an $\epsilon$ be given. 
From \Cref{cor_BkandBkinversebounded} we infer that for all $k\in\N_0$ we have  
\begin{equation}\label{eq_estgradJinlcp}
	\norm{d_k} \leq C\norm{\nabla f(x_k)}^{1-(2m+1)c_2}.
\end{equation}			
The exponent in \cref{eq_estgradJinlcp} is positive because of the assumption $c_2<(2m+1)^{-1}$,
hence $\norm{d_k}<\epsilon$ whenever $\norm{x_k-\xopt}<\delta'$ for some sufficiently small $\delta'>0$, where we used the continuity of $\nabla f$ at $\xopt$ and $\nabla f(\xopt)=0$, which holds due to \Cref{thm_clusterpointsarestationar}. 
If Algorithm~\ref{alg_hybrid} uses the Armijo rule with backtracking, then $\alpha_k\leq 1$ for all $k$, 
so the desired implication follows. 
In the remainder of Part~1 we can therefore assume that all $\alpha_k$, $k\in\N_0$, satisfy the Wolfe--Powell conditions.
The strong convexity of $f$ in $\nbhd$ implies the existence of $\mu>0$ such that 
\begin{equation*}
f(x)-f(\xopt)\leq \frac{1}{2\mu}\norm{\nabla f(x)}^2
\end{equation*}
for all $x\in\nbhd$. 
In particular, this holds for $x=x_k$ whenever $x_k\in\Ballop{\epsilon}(\xopt)$. 
Moreover, the Armijo condition \cref{eq_armijocond} holds for all step sizes $\alpha_k$, $k\in\N_0$. 
Together, we have for all $k\in\N_0$ with $x_k\in\Ballop{\epsilon}(\xopt)$ 
\begin{equation*}
\begin{split}
	\alpha_k\sigma \nabla f(x_k)^T H_k \nabla f(x_k)
	&\leq f(x_k)-f(x_{k+1})\\
	&\leq f(x_k)-f(\xopt) 
	\leq \frac{1}{2\mu}\norm{\nabla f(x_k)}^2,
\end{split}
\end{equation*}	
where we used that $f(x_{k+1})\geq f(\xopt)$ as $(f(x_k))$ is monotonically decreasing. 
Thus, for these $k$
\begin{equation*}
\alpha_k \leq \frac{\norm{H_k^{-1}}}{2\sigma\mu}.
\end{equation*}	
From \Cref{lem_BkandBkinversebounded} we obtain that 
$\norm{H_k^{-1}} \leq \hat C\omega_k^{-1}$ for all $k\in\N_0$, where $\hat C>0$ is a constant. 
Decreasing $\delta'$ if need be, we may assume $\delta'\leq\epsilon$ and 
$\omega_k^{-1}=c_1^{-1}\norm{\nabla f(x_k)}^{-c_2}$ for all $k\in\N_0$ with $x_k\in\Ballop{\delta'}(\xopt)$.
Combining this with \cref{eq_estgradJinlcp} we obtain for all $k\in\N_0$ with $x_k\in\Ballop{\delta'}(\xopt)$ that
\begin{equation*}
\begin{split}
	\norm{s_k} = \alpha_k\norm{d_k} \leq \frac{\norm{H_k^{-1}}}{2\sigma\mu}\norm{d_k}
	& \leq \frac{C\hat C}{2\sigma\mu} \omega_k^{-1} \norm{\nabla f(x_k)}^{1-(2m+1)c_2}\\
	& = \frac{C\hat C}{2\sigma\mu c_1}\norm{\nabla f(x_k)}^{1-2(m+1)c_2}.
\end{split}
\end{equation*}
The choice of $c_2$ in \Cref{ass_linconv}~2) implies $1-2(m+1)c_2>0$. 
Thus, after decreasing $\delta'$ if need be, 
there holds $\norm{s_k}<\epsilon$ for all $k\in\N_0$ with $x_k\in\Ballop{\delta'}(\xopt)$.
This finishes the proof of Part~1.\\
\noindent\textbf{Part 2: Convergence of the entire sequence $\mathbf{(x_k)}$.}\\		
Let $\epsilon'>0$ be so small that $\Ballop{\epsilon'}(\xopt)\subset\nbhd$. 
We have to show that there is $\bar k\geq 0$ such that $x_k\in\Ballop{\epsilon'}(\xopt)$ for all $k\geq \bar k$.  
Due to Part~1 we find a positive $\delta'$ such that 
for all $k\in\N_0$ the implication $x_k\in\Ballop{\delta'}(\xopt)\Rightarrow\norm{s_k}<\epsilon'/2$ holds true.
Decreasing $\delta'$ if need be, we can assume that $\delta'\leq\epsilon'/2$. 
It then follows that $x_{k+1}\in\Ballop{\epsilon'}(\xopt)$ for all $k\in\N_0$ with $x_k\in\Ballop{\delta'}(\xopt)$.
Next we use that the $\mu$-strongly convex function $f\vert_{\Ballop{\epsilon'}(\xopt)}$ satisfies the growth condition
\begin{equation}\label{eq_strconvgrc}
f(\xopt) + \frac{\mu}{2}\norm{x-\xopt}^2 \leq f(x)
\end{equation}
for all $x$ in $\Ballop{\epsilon'}(\xopt)$. 
Let $U:=\Ballop{\epsilon'}(\xopt)\setminus\Ballop{\delta'}(\xopt)$. 
Due to \cref{eq_strconvgrc} we have $f(x)-f(\xopt)\geq (\delta')^2\mu/2$ for all $x\in U$.
Since $(f(x_k))$ converges by \Cref{lem_welldef}, we find $\hat k$ such that $f(x_k)-f(\xopt)<(\delta')^2\mu/2$ for all $k\geq\hat k$, hence $x_k\not\in U$ for all $k\geq \hat k$.
Selecting $\bar k\geq \hat k$ such that $x_{\bar k}\in\Ballop{\delta'}(\xopt)$, 
we obtain that $x_{\bar k+1}\in\Ballop{\epsilon'}(\xopt)$ and 
$x_{\bar k+1}\not\in U$, thus $x_{\bar k+1}\in\Ballop{\delta'}(\xopt)$. 
By induction we infer that $x_k\in\Ballop{\delta'}(\xopt)$ for all $k\geq \bar k$, which concludes the proof as $\delta'\leq\epsilon'/2$.
\end{proof}

Next we show that $(\norm{H_k})$ and $(\norm{H_k^{-1}})$ are bounded.

\begin{lemma}\label{lem_Bkandinverseboundedforconvtopssosc}
Let \Cref{ass_basic} hold and let $(x_k)$ be generated by Algorithm~\ref{alg_hybrid}. 
Suppose there are $\xopt\in\CX$ and a convex neighborhood $\nbhd$ of $\xopt$ such that $\lim_{k\to\infty}x_k=\xopt$, $f\vert_{\nbhd}$ is strongly convex
and $\nabla f\vert_{\nbhd}$ is Lipschitz. Then $(\norm{H_k})$ and $(\norm{H_k^{-1}})$ are bounded. 
\end{lemma}

\begin{proof}
Let $\bar k$ be such that $x_k\in\nbhd$ for all $k\geq\bar k$. Since $f\vert_{\nbhd}$ is strongly convex, there is $\mu>0$ such that $\nabla f$ is $\mu$-strongly monotone in $\nbhd$, i.e.,
\begin{equation*}
\bigl[ \nabla f(\hat x)-\nabla f(x) \bigr]^T (\hat x-x) \geq \mu \norm{\hat x-x}^2
\end{equation*}
for all $x,\hat x\in\nbhd$. By inserting $\hat x=x_{j+1}$ and $x=x_j$ we infer that 
\begin{equation}\label{eq_kestwp}
y_j^T s_j\geq \mu\norm{s_j}^2 \qquad\text{ and }\qquad
y_j^T s_j\geq \frac{\mu}{L^2} \norm{y_j}^2
\end{equation} 
for all $j\geq\bar k$, where the second estimate follows from the first by the Lipschitz continuity of $\nabla f$ in $\nbhd$.
Note that $y_j^T s_j>0$ for all $j\geq\bar k$, so in any iteration $k\geq\bar k$ the pair $(s_k,y_k)$ enters the storage, cf. Line~\ref{line_acceptanceofysforstorage} and Line~\ref{line_skykappend} in Algorithm~\ref{alg_hybrid}. 
Therefore, at the beginning of iteration $k\geq\bar k+m$ we have $\CI=\{k-m,k-m+1,\ldots,k-1\}$ 
for the index set of the storage (with $\CI=\emptyset$ if $m=0$), 
and consequently \cref{eq_kestwp} holds for all pairs $(s_j,y_j)$ in the storage whenever the iteration counter $k$ is sufficiently large. 
In view of \Cref{lem_generalresultonnormbounds} it only remains to prove that $H_k^{(0)}=\gamma_k I$ and its inverse are bounded independently of $k$, i.e., that $(\gamma_k)$ and $(\gamma_k^{-1})$ are bounded from above. 
Since $\lim_{k\to\infty}\norm{\nabla f(x_k)}=0$ by \Cref{thm_clusterpointsarestationar}, we infer that $\lim_{k\to\infty}\omega_k=0$. 
We now show that $(\gamma_k^\Hy)$ is bounded away from zero and that $(\gamma_k^\Hs)$ is bounded from above for sufficiently large $k$. This implies that $[\gamma_k^\Hy,\gamma_k^\Hs]\cap[\omega_k,\omega_k^{-1}]=[\gamma_k^\Hy,\gamma_k^\Hs]$ for all sufficiently large $k$, in turn 
showing that $(\gamma_k)$ and $(\gamma_k^{-1})$ are bounded, cf. Line~\ref{line_choicetauk} in Algorithm~\ref{alg_hybrid}. 
As we have already established, there holds $y_k^T s_k>0$ for all $k\geq\bar k$. By Line~\ref{line_acceptanceofysforstorage} we thus deduce that for all $k\geq\bar k$, $\gamma_{k+1}^\Hy$ and $\gamma_{k+1}^\Hs$ are computed according to Definition~\ref{def_tau_B}. 
Together with \cref{eq_kestwp} we readily obtain that
\begin{equation*}
\gamma_{k+1}^\Hy = \frac{y_k^T s_k}{\norm{y_k}^2} \geq \frac{\mu}{L^2}
\qquad\text{ and }\qquad
\gamma_{k+1}^\Hs = \frac{\norm{s_k}^2}{y_k^T s_k} \leq \frac{1}{\mu},
\end{equation*}
both valid for all $k\geq\bar k$. This concludes the proof.
\end{proof}

\subsubsection{Transition to classical \LBFGS}\label{sec_transitiontoLBFGS}

Before we prove the linear convergence of Algorithm~\ref{alg_hybrid} in \Cref{thm_linconv}, let us draw some conclusions from 
\Cref{lem_Bkandinverseboundedforconvtopssosc} that shed more light on Algorithm~\ref{alg_hybrid}.
In particular, this will enable us to establish in \Cref{thm_lbfgsmislbfgs} that Algorithm~\ref{alg_hybrid} turns into Algorithm~\ref{alg_LBFGS} when approaching a minimizer that satisfies sufficient optimality conditions. 

We start by noting that in the situation of \Cref{lem_Bkandinverseboundedforconvtopssosc}, 
eventually each new update pair $(s_k,y_k)$ enters the storage, all stored pairs are used to compute~$H_k$,
and any $\gamma_k\in [\gamma_k^\Hy,\gamma_k^\Hs]$ can be chosen.  

\begin{lemma}\label{lem_noone}
Under the assumptions of \Cref{lem_Bkandinverseboundedforconvtopssosc} there exists $\bar k\in\N_0$ such that 
when arriving at Line~\ref{line_tlr} of Algorithm~\ref{alg_hybrid} in iteration $k\geq\bar k+m$ there holds 
$\CI = \{k-m,k-m+1,\ldots,k-1\}$ (with $\CI=\emptyset$ if $m=0$), 
i.e., the storage $\{(s_j,y_j)\}_{j\in\CI}$ consists of the $m$ most recent update pairs.
Moreover, all pairs are used for the computation of $H_k$ in Line~\ref{line_tlr} for $k\geq\bar k$. 
Also, $\bar k$ can be chosen such that for all $k\geq\bar k$ we have $[\gamma_k^\Hy,\gamma_k^\Hs]\cap [\omega_k,\omega_k^{-1}]=[\gamma_k^\Hy,\gamma_k^\Hs]$ 
in Line~\ref{line_choicetauk} of Algorithm~\ref{alg_hybrid}.
\end{lemma}

\begin{proof}
In the proof of \Cref{lem_Bkandinverseboundedforconvtopssosc} we already argued that $\CI = \{k-m,\ldots,k-1\}$ for $k\geq \bar k+m$. 
Regarding the computation of $H_k$ we recall that \cref{eq_kestwp} holds for all $j\geq\bar k$,
so $q(s_j,y_j)\geq \min\{\mu,\mu/L^2\}$ for all these $j$ and hence for all pairs in the storage in iteration $k\geq\bar k+m$.
We have also shown in the proof of \Cref{lem_Bkandinverseboundedforconvtopssosc} that $\lim_{k\to\infty}\omega_k=0$.	
Thus, for all large $k$ there holds
$q(s_j,y_j)\geq \omega_k$ for all $j\in\CI$, so all pairs in the storage are used for the computation of $H_k$, cf.~Line~\ref{line_tlr}.
We also recall from the proof of \Cref{lem_Bkandinverseboundedforconvtopssosc} that $[\gamma_k^\Hy,\gamma_k^\Hs]\cap[\omega_k,\omega_k^{-1}]=[\gamma_k^\Hy,\gamma_k^\Hs]$ for all sufficiently large $k$.		
\end{proof}

We can now argue that Algorithm~\ref{alg_hybrid} turns into Algorithm~\ref{alg_LBFGS} close to $\xopt$. More precisely, 
if in iteration $k$ of Algorithm~\ref{alg_hybrid}, where $k$ is sufficiently large, we took a snapshot of the storage $\CI$ 
and initialized Algorithm~\ref{alg_LBFGS} with $x_k$ and that storage, then 
the iterates generated subsequently by these algorithms would be identical (and so would the storages, the step sizes, etc.), regardless of the choice of constants in Algorithm~\ref{alg_hybrid}. 
Of course, this assumes that the seed matrices and the step sizes are selected in the same way in both algorithms, e.g.,
$\alpha_k$ is determined according to the Armijo condition with the same constant $\sigma$ and using identical backtracking mechanisms. 
For ease of presentation let us assume that both algorithms choose
$H_{k}^{(0)}=\gamma_{k}^\Hy I$ whenever possible, i.e., Algorithm~\ref{alg_LBFGS} makes this choice if $\gamma_{k}^\Hy>0$ while
Algorithm~\ref{alg_hybrid} makes this choice if $\gamma_{k}^\Hy\in[\omega_k,\omega_k^{-1}]$. 	
To distinguish the quantities generated by the two algorithms, we indicate those of Algorithm~\ref{alg_LBFGS} by a \emph{hat}, e.g., we write $(\hat x_k)$ for the iterates of \ref{alg_LBFGS} and $(x_k)$ for the iterates of \ref{alg_hybrid}.

\begin{theorem}\label{thm_lbfgsmislbfgs}
Let \Cref{ass_linconv} hold and let $(x_k)$ be generated by Algorithm~\ref{alg_hybrid}. 
Let $(x_k)$ have a cluster point $\xopt$ with a convex neighborhood $\nbhd$ such that $f\vert_{\nbhd}$ is strongly convex and $\nabla f\vert_{\nbhd}$ is Lipschitz. Then $(x_k)$ converges to $\xopt$ and 
Algorithm~\ref{alg_hybrid} eventually turns into Algorithm~\ref{alg_LBFGS}. More precisely, 
consider the sequence $(\hat x_k)$ generated by Algorithm~\ref{alg_LBFGS} with starting point $\hat x_0:=x_{\bar k}$ for some $\bar k\geq m$ using the initial storage $\hat\CI:=\{(s_j,y_j)\}_{j=\bar k-m}^{\bar k-1}$.
Suppose that for all $k\geq 0$, Algorithm~\ref{alg_LBFGS} selects $\hat\alpha_k$ in the same way as Algorithm~\ref{alg_hybrid} selects 
$\alpha_{\bar k+k}$. If $\bar k$ is sufficiently large, then for any $k\in\N_0$ we have 
$x_{\bar k+k}=\hat x_k$, $H_{\bar k+k}=\hat H_k$, $\CI_{\bar k+k}=\hat\CI_k$ and $\alpha_{\bar k + k} = \hat\alpha_k$.
Moreover, if $\bar k$ is sufficiently large, then for all $k\geq\bar k$ the storage $\{(s_j,y_j)\}_{j\in\CI}$ consists of the $m$ most recent pairs, they are all used for the computation of $H_k$, and any $\gamma_k\in [\gamma_k^\Hy,\gamma_k^\Hs]$ is accepted in Line~\ref{line_choicetauk} of Algorithm~\ref{alg_hybrid}.  
\end{theorem}

\begin{proof}
By \Cref{lem_aux}, $(x_k)$ converges to $\xopt$, so the assumptions of \Cref{lem_Bkandinverseboundedforconvtopssosc} are satisfied. All claims then follows by induction over $k$ using \Cref{lem_noone}. 
\end{proof}

\begin{remark}\label{rem_LBFGSMagreeswithLBFGSforsmallconstants}
We conclude the study of the relationship between \ref{alg_LBFGS} and \ref{alg_hybrid} with the following observations. 
\begin{enumerate}
\item[1)] Let $\xopt$ be a stationary point having a neighborhood $\nbhd$ in which $f$ is strongly convex and $\nabla f$ is Lipschitz. It can be shown 
similarly to \Cref{lem_noone} that for any choice of constants $c_0,c_1,c_2$ in Algorithm~\ref{alg_hybrid} there is a convex neighborhood $\hat\nbhd\subset\nbhd$ such that
when initialized with any $x_0\in\hat\nbhd$, Algorithm~\ref{alg_LBFGS} and Algorithm~\ref{alg_hybrid} are identical (assuming that they choose $H_k^{(0)}$ and $\alpha_k$ in the same way for all $k$).
This shows that when initialized sufficiently close to a point that satisfies sufficient optimality conditions, 
Algorithm~\ref{alg_LBFGS} and Algorithm~\ref{alg_hybrid} agree. 
\item[2)] Let $\xopt$ be a stationary point having a neighborhood $\nbhd$ in which $f$ is strongly convex and $\nabla f$ is Lipschitz. 
Suppose that Algorithm~\ref{alg_LBFGS} has generated iterates $(x_k)$ that converge to $\xopt$.
Since there are only finitely many iterates outside of $\nbhd$, it is not difficult to argue that \ref{alg_hybrid} agrees entirely with Algorithm~\ref{alg_LBFGS} (assuming that they choose $H_k^{(0)}$ and $\alpha_k$ in the same way for all $k$) provided that one of the constants $c_0,c_1$ is chosen sufficiently small. However, the required value of these constants depends on the starting point $x_0$ and making
it dependent on the associated level set $\Omega$ but not on $x_0$ would require additional assumptions. An example for such assumptions is given in 3).
This shows that if Algorithm~\ref{alg_LBFGS} converges to a point that satisfies sufficient optimality conditions, then it agrees with Algorithm~\ref{alg_hybrid} for sufficiently small $c_0$ or $c_1$. We mention again that we allow $m=0$ in these algorithms, so this also holds for the \BBB~method. 
We infer further that for appropriate choices of its parameters, \LBFGS~with the classical cautious updating from \cite{LF01} also agrees with \ref{alg_LBFGS} in that situation. 
\item[3)] It is not difficult to show that if $f$ is strongly convex in $\Omega$ with Lipschitz continuous gradients, then
Algorithm~\ref{alg_LBFGS} and Algorithm~\ref{alg_hybrid} agree entirely for any starting point $x_0\in\Omega$, provided one of the constants $c_0,c_1$ in \ref{alg_hybrid} is chosen sufficiently small. 
Here, the required values of $c_0,c_1$ depend on $\Omega$, but not on $x_0$.
\end{enumerate}
\end{remark}

\subsubsection{Main result}

As the main result of this work we prove that if Algorithm~\ref{alg_hybrid} generates a sequence $(x_k)$ with a cluster point near which $f$ is strongly convex and near which it has Lipschitz continuous gradients, then the entire sequence converges to that point and the convergence is linear.
Note that \Cref{thm_lbfgsmislbfgs} applies in this situation, guaranteeing eventually that the scaling factor $\gamma_k$ of the seed matrix can be chosen appropriately and that the $m$ most recent update pairs are stored and used for the computation of $H_k$. 

\begin{theorem}\label{thm_linconv}
Let \Cref{ass_linconv} hold and let $(x_k)$ be generated by Algorithm~\ref{alg_hybrid}. 
Suppose that $(x_k)$ has a cluster point $\xopt$, that $f$ is strongly convex in a convex neighborhood $\nbhd_1$ of $\xopt$, and that $\nabla f$ is Lipschitz in a neighborhood $\nbhd_2\subset\nbhd_1$ of $\xopt$.
Then 
\begin{itemize}
\item $\xopt$ is an isolated local minimizer of $f$; 
\item $(f(x_k))$ converges q-linearly to $f(\xopt)$;
\item $(x_k)$ converges $l$-step q-linearly to $\xopt$ for any sufficiently large $l$; 
\item $(\nabla f(x_k))$ converges $l$-step q-linearly to zero
for any sufficiently large $l$. 
\end{itemize}
More precisely, let $f$ be $\mu$-strongly convex in $\nbhd_1$ and let $\nabla f$ be $L$-Lipschitz in $\nbhd_2$.
Then the numbers $\bar k_i:=\min\{\hat k: x^k\in\nbhd_i\;\forall k\geq\hat k\}$ are well-defined for $i=1,2$, and setting
\begin{equation*}
\nu_k := 1-\frac{2\sigma\alpha_k\mu}{\norm{H_k^{-1}}}	\qquad\forall k\in\N_0
\end{equation*}
we have $\sup_k\nu_k<1$ and $(f(x_k))$ converges q-linearly to $f(\xopt)$ in that
\begin{equation}\label{eq_linconvJ}
f(x_{k+1})-f(\xopt)
\leq \nu_k\bigl[f(x_k)-f(\xopt)\bigr] \qquad \forall k\geq\bar k_1.
\end{equation}
Furthermore, for any $l\in\N_0$ and all $k\geq\bar k_2$ there hold
\begin{equation}\label{eq_msteplinearconv}
\norm{x_{k+l}-\xopt} \leq \sqrt{\kappa\nu^l}\; \norm{x_k-\xopt}\quad\text{and}\quad
\norm{\nabla f(x_{k+l})}\leq \kappa\sqrt{\nu^l} \; \norm{\nabla f(x_k)},
\end{equation}
where $\kappa:=L/\mu$ and $\nu:=\sup_{k\geq\bar k_2}\nu_k$. 
\end{theorem}	

\begin{proof}
\textbf{Part 1: Preliminaries}\\
\Cref{thm_clusterpointsarestationar} implies that 
$\nabla f(\xopt)=0$. 
Together with the strong convexity of $f$ it follows that 
$\xopt$ is the unique local and global minimizer of $f$ in $\nbhd_1$, hence isolated. 
Moreover, \Cref{lem_aux} implies that $(x_k)$ converges to $\xopt$ and \Cref{lem_Bkandinverseboundedforconvtopssosc} thus yields
that $(\norm{H_k})$ and $(\norm{H_k^{-1}})$ are bounded.\\
\textbf{Part 2: Q-linear convergence of $\mathbf{(f(x_k))}$}\\
Setting $h_k:=\norm{H_k^{-1}}^{-1}>0$ we use the Armijo condition and $d_k = -H_k\nabla f(x_k)$ to infer for all $k$ 
\begin{equation*}
\begin{split}
	f(x_{k+1})-f(\xopt)
	& = f(x_{k+1})-f(x_k)+f(x_k)-f(\xopt)\\
	&\leq \sigma\alpha_k\nabla f(x_k)^T d_k + f(x_k)-f(\xopt)\\
	&\leq -\sigma\alpha_k h_k\norm{\nabla f(x_k)}^2 + f(x_k)-f(\xopt).
\end{split}
\end{equation*}
The $\mu$-strong convexity of $f$ implies 
\begin{equation}\label{eq_KLineq}
f(x)-f(\xopt)\leq \frac{1}{2\mu} \norm{\nabla f(x)}^2	
\end{equation}
for all $x\in\nbhd_1$. Thus, for all $k\geq\bar k_1$ we have 
\begin{equation*}
f(x_{k+1})-f(\xopt)
\leq \left(1-2\sigma\alpha_k h_k\mu\right) \bigl[f(x_k)-f(\xopt)\bigr],
\end{equation*}
which proves \cref{eq_linconvJ}. 
The boundedness of $(\norm{H_k^{-1}})$ implies $\inf_k h_k>0$. 
Moreover, there is $\alpha>0$ such that $\alpha_k\geq\alpha$ for all $k$, as is well-known both for the Wolfe--Powell conditions \cite[(3.6)]{S16} and for backtracking with the Armijo condition \cite[proof of Lemma~4.1]{BN89} 
(here we need the Lipschitz continuity of $\nabla f$ in $\nbhd_2$ and the boundedness of $(\norm{H_k})$). 
Together, we infer that $\sup_{k}\nu_k<1$.\\
\textbf{Part 3: Convergence of $\mathbf{(x_k)}$ and $\mathbf{(\nabla f(x_k))}$}\\
The strong convexity yields the validity of \cref{eq_strconvgrc} for all $x\in\nbhd_1$. Together with \cref{eq_linconvJ} 
and the Lipschitz continuity of $\nabla f$ in $\nbhd_2$ we estimate for all $k\geq\bar k_2$ and all $l\in\N_0$ 
\begin{equation}\label{eq_strconvconvest}
\begin{split}
	\frac{\mu}{2}\norm{x_{k+l}-\xopt}^2 
	&\leq f(x_{k+l})-f(\xopt)\\
	&\leq \left[\prod_{j=k}^{k+l-1}\nu_j\right] \bigl[f(x_k)-f(\xopt)\bigr]
	\leq \nu^l\frac{L}{2} \norm{x_k-\xopt}^2,
\end{split}
\end{equation}
where we used that $\nu_j\leq\nu=\sup_{k\geq\bar k_2}\nu_k$ for all $j\geq\bar k_2$. 
Evidently, this implies the left estimate in \cref{eq_msteplinearconv}. 
Since we have established in Part~2 of the proof that $\nu<1$, there is a minimal $l^\ast\in\N$ such that $\sqrt{\kappa\nu^{l^\ast}}\in(0,1)$. Hence, 
$\sqrt{\kappa\nu^l}\in(0,1)$ for any $l\geq l^\ast$, so the left estimate in \cref{eq_msteplinearconv} indeed shows the $l$-step q-linear convergence of $(x_k)$ for any $l\geq l^\ast$.

For $(\nabla f(x_k))$ we infer from the Lipschitz continuity, \cref{eq_strconvconvest} and \cref{eq_KLineq} for all $k\geq\bar k_2$ and all $l\in\N_0$ 
\begin{equation*}
\norm{\nabla f(x_{k+l})}^2 \leq L^2\norm{x_{k+l}-\xopt}^2 
\leq \frac{2L^2}{\mu} \nu^l\bigl[f(x_k)-f(\xopt)\bigr]
\leq \kappa^2 \nu^l \norm{\nabla f(x_k)}^2.
\end{equation*}
This proves the right estimate in \cref{eq_msteplinearconv}. The $l$-step q-linear convergence follows as for $(x_k)$. 
\end{proof}

\begin{remark}\label{rem_thmlinconv}
\phantom{linebreak}
\begin{enumerate}
\item[1)] As mentioned in \Cref{rem_LBFGSMagreeswithLBFGSforsmallconstants}~2) and 3), 
it can be proven that if $f$ is strongly convex with $\nabla f$ Lipschitz that 
Algorithm~\ref{alg_hybrid} agrees with classical \LBFGS~if the constant $c_0$ is sufficiently small, and this also holds
if Algorithm~\ref{alg_LBFGS} generates a sequence $(x_k)$ that converges to some $\xopt$ near which $f$ is strongly convex and has Lipschitz continuous gradients. 
Thus, in these cases \Cref{thm_linconv} applies not only to Algorithm~\ref{alg_hybrid} but also holds for classical \LBFGS. 
In the first case \Cref{thm_linconv} extends and improves the standard result \cite[Theorem~7.1]{LN89} for classical \LBFGS~that shows only r-linear convergence of $(x_k)$, does not include a rate of convergence for $(\nabla f(x_k))$, assumes that $f$ is twice continuously differentiable with bounded second derivatives, and covers only $\CX=\R^n$.
\item[2)] From \cref{eq_linconvJ} and the fact that $f(x_{k+1})-f(\xopt)<f(x_k)-f(\xopt)$ for all $k\leq \bar k_1$ it follows that
there is $\bar\nu\in(0,1)$ such that $f(x_{k+1})-f(\xopt)\leq\bar\nu\left[f(x_k)-f(\xopt)\right]$ for all $k\in\N_0$.
Hence, Algorithm~\ref{alg_hybrid} generally consists of two phases: First, in the \emph{global phase}, lasting from iteration $k=0$ to at most iteration $k=\bar k_2$, the objective function decays q-linearly but we have no control over the errors $\norm{x_k-\xopt}$ and $\norm{\nabla f(x_k)}$. Second, in the \emph{local phase}, starting at iteration $k=\bar k_2$ or earlier, the errors $\norm{x_k-\xopt}$ and $\norm{\nabla f(x_k)}$ 
become $l$-step q-linearly convergent for any sufficiently large~$l$. 
Somewhere in between, specifically starting at or before iteration $k=\bar k_1$, the errors in the objective start to satisfy \cref{eq_linconvJ}. 
Note that the behavior of the algorithm in the local phase is closely related to the \emph{local condition number} $\kappa=L/\mu$
\cite{N18}. 
If $\Omega$ is convex and $f$ is strongly convex with Lipschitz gradient in $\Omega$, then there is no global phase since $\bar k_2=0$. As in 1) this also holds for classical \LBFGS.
\item[3)] The estimates in \cref{eq_msteplinearconv} are still meaningful if $l$ is such that $\sqrt{\kappa\nu^l}$, respectively, $\kappa\sqrt{\nu^l}$ is larger than one
because they limit the \emph{increase} of $\norm{x_{k+l}-\xopt}$ in comparison to $\norm{x_{k}-\xopt}$, respectively, of $\norm{\nabla f(x_{k+l})}$ in comparison to $\norm{\nabla f(x_k)}$. For $l=1$ we infer that there is a constant $C>0$ such that 
the quotients $\norm{x_{k+1}-\xopt}/\norm{x_k-\xopt}$ and $\norm{\nabla f(x_{k+1})}/\norm{\nabla f(x_{k})}$ are bounded from above by $C$. This is generally not true for r-linear convergence.
\item[4)] Note that $l$-step q-linear convergence for some $l\in\N$ does not generally imply $j$-step q-linear convergence for $j>l$. For example, let $0<a<b<1$ and consider the sequence $\tau_{2n-1}:=a^n$, $\tau_{2n}:=b^n$. This sequence converges $l$-step q-linearly if and only if $l$ is even.
\end{enumerate}
\end{remark}
		

\section{Numerical experiments}\label{sec_experiments}

Before we present the numerical experiments that follow, let us stress that in all experiments, including those that are not reported, we have consistently found that \ref{alg_hybrid} agrees \emph{entirely} with the classical \LBFGS/\BBB~method Algorithm~\ref{alg_LBFGS} for moderately small values of $c_0$. 
This suggests that \LBFGS~is inherently cautious, which has also been observed for BFGS with cautious updates \cite{LiFu01}.
We may also recall from \Cref{rem_LBFGSMagreeswithLBFGSforsmallconstants} that for any finite set of starting points such that the classical method only converges to points that satisfy sufficient optimality conditions, there is $c_0>0$ such that the two methods agree entirely. 
Additionally, let us note that whenever \ref{alg_hybrid} agrees entirely with \ref{alg_LBFGS}, then
these two algorithms are also identical to \LBFGS/\BBB~with standard cautious updating \cite{LF01} for appropriate parameters. 

In the following experiments, we choose $c_0$ so small that \ref{alg_hybrid} and \ref{alg_LBFGS} agree. 
Since the literature contains ample numerical experiments for \LBFGS, 
we consider only three test problems for \ref{alg_hybrid} to illustrate some of the new theoretical results. 
Recent works that include numerical experiments for \LBFGS~are, for instance, 
\cite{A22,KS23}.

The values of $c_0,c_1,c_2$ are specified in \Cref{tab_paramvalues}.
We use $\gamma_k=\gamma_k^\Hy$ for all $k$, cf. Definition~\ref{def_tau_B}.
The computation of $-H_k\nabla f(x_k)$ in Line~\ref{line_tlr} is realized in a matrix free way through the two-loop recursion \cite[Algorithm~7.5]{NW06}. 
Algorithm~\ref{alg_hybrid} terminates when it has generated an $x_k$ that satisfies $\norm{\nabla f(x_k)}\leq 10^{-9}$ in the first and third example, respectively, $\norm{\nabla f(x_k)}\leq 10^{-5}$ in the second. 
Regarding line search strategies, we use Armijo with backtracking by the fixed factor $\beta:=\beta_1=\beta_2=\frac12$ in all examples. 
In the first and third example we additionally apply the well known \MT~line search \cite{MT94} from \textsc{Poblano} \cite{Poblano}, which ensures that the strong Wolfe--Powell conditions are satisfied. 
In the second example we replace the \MT~line search by another one; details are discussed in that example. 
The parameter values of the line searches are included in \Cref{tab_paramvalues}.
The experiments are conducted in \textsc{MATLAB}~2023b. 
The code for the first example is available on \textsc{arXiv} with the \href{https://doi.org/10.48550/arXiv.2401.03805}{preprint of this article}. It includes \ref{alg_hybrid} and \ref{alg_LBFGS}. 

\begin{table}
	\small
	\centering
	\begin{tabular}{c|c|c|c|c|c|c|c|c|c|c}
		\toprule
		\multicolumn{3}{c|}{Algorithm~\ref{alg_hybrid}} & \multicolumn{2}{c|}{Armijo} & \multicolumn{6}{c}{\WP~\& Moré--Thuente (Poblano)} \\
		\midrule
		$c_0$ & $c_1$ & $c_2$ & $\sigma$ & $\beta$ & 
		$\sigma$ & $\eta$ 
		& $maxfev$ & $stpmax$ & $stpmin$ & $xtol$\\
		\midrule
		$10^{-4}$ & $1$ & $2m+3$ & $10^{-4}$ & $0.5$ & $10^{-4}$ & $0.9$ &
		$20$ & $1000$ & $0$ & $10^{-7}$\\
		\bottomrule
	\end{tabular}
	\caption{Parameter values for Algorithm~\ref{alg_hybrid} and the line searches}
	\label{tab_paramvalues}
\end{table}

Next we define some quantities for the evaluation of the numerical results. 
Suppose that Algorithm~\ref{alg_hybrid} terminates for $k=K$ in the modified Line~\ref{line_term}. 
It has then generated iterates $x_0,x_1,\ldots,x_K$, step sizes $\alpha_0,\alpha_1,\ldots,\alpha_{K-1}$ and has taken $K$ iterations if we count $k=0$ as the first iteration and do not count the incomplete iteration for $k=K$.
The number of iterations in which $(s_k,y_k)$ is added to the storage is $\CP:=\lvert\{k\in\{0,\ldots,K-1\}: y_k^T s_k>0 \}\rvert$.
Note that the maximal value of $\CP$ is $K$ and that for the \MT~line search there holds $\CP=K$. 
By $\alpha_{\min}$ and $\alpha_{\max}$ we denote the smallest, respectively, largest step size that is used during the course of Algorithm~\ref{alg_hybrid}.
Moreover, we let
\begin{equation*}
		Q_{f}:=\max_{1\leq k\leq K}\left\{\frac{f(x_k)-\fopt}{f(x_{k-1})-\fopt}\right\},
		\quad
		Q_x:=\max_{1\leq k\leq K}\left\{\frac{\norm{x_k-\xopt}}{\norm{x_{k-1}-\xopt}}\right\}, \quad
		Q_{g}:=\max_{1\leq k\leq K}\left\{\frac{\norm{\nabla f(x_k)}}{\norm{\nabla f(x_{k-1})}}\right\}
\end{equation*}
denote the maximal q-factors of the respective sequences.
To assess the asymptotic of the q-factors, we also consider 
a variant in which $\max_{1\leq k\leq K}$ is replaced by $\max_{K-2\leq k\leq K}$, 
i.e., where the maximum is only taken wrt. the final three quotients.
This variant is denoted by $Q_f^3$, $Q_x^3$ and $Q_g^3$, respectively.

\subsection{Example~1: The Rosenbrock function}

As a classical example we consider the Rosenbrock function $f:\R^2\rightarrow\R$, $f(x):=(1-x_1)^2+100(x_2-x_1^2)^2$,
with unique global minimizer $\xopt=(1,1)^T$ that is also the unique stationary point of $f$. 
It is straightforward to confirm that $f$ is strongly convex in the square $[-0.5,1.4]^2$ surrounding $\xopt$. 
Since every level set of $f$ is compact, $\nabla f$ is Lipschitz continuous in $\Omega$ regardless of the starting point.
However, $f$ is not convex in the level set associated to the starting point $x_0=(-1.2,1)^T$ that we use, so there is no result available that guarantees convergence of the classical \LBFGS~method to $\xopt$ from this starting point.
For \ref{alg_hybrid}, in contrast, \Cref{thm_clusterpointsarestationar} shows that $(x_k)$ converges to $\xopt$ for any $x_0\in\R^2$,
and \Cref{thm_linconv} further implies that if $c_2<1/(2m+2)$, 
convergence is either finite or at least linear.
In addition, \Cref{thm_lbfgsmislbfgs} ensures that \ref{alg_hybrid} turns into \LBFGS~as $\xopt$ is approached.
\Cref{tab_TP1} and \Cref{fig_TP1} show the numerical results of Algorithm~\ref{alg_hybrid}.

\begin{table}
	\scriptsize
	\centering
	\begin{tabular}{cc|c|c|c|c|c|c|c|c|c}
		& & \#it  & $\#f$ &  $\CP$ & $\alpha=1$ & $\alpha_{\max}$ & $\alpha_{\min}$ & $Q_f/Q_f^3$ & $Q_x/Q_x^3$ & $Q_g/Q_g^3$ \\
		\hline
		Armijo & $(m=0)$ & 82 & 129 & 78 & 62 & $1$ & $5\mathrm{e}{-4}$ & $0.9993$/$0.9992$ & $1.03$/$0.9996$ & $40$/$0.9996$ \\
		\MT & ($m=0$)
		& 4121 & 8252 & 4121 & 2057 & $341$ & $1\mathrm{e}{-3}$ & $0.9992$/$0.998$ & $1.02$/$0.9996$ & $15$/$2$ \\
		\hline
		Armijo & $(m=1)$ 
		& 90 & 154 & 89 & 71 & 1 & $1\mathrm{e}{-3}$ &  $0.9992$/$0.75$ & $2$/$1.04$ & $6$/$1.6$ \\
		\MT & ($m=1$)
		& 46 & 84 & 46 & 21 & $341$ & $1\mathrm{e}{-3}$ & $0.9992$/$0.37$ & $2$/$0.61$ & $8$/$0.61$ \\
		\hline
		Armijo & $(m=2)$ 
		& 42 & 90 & 42 & 29 & 1 & $1\mathrm{e}{-3}$ &  $0.9992$/$0.19$ & $1.02$/$0.43$ & $12$/$0.44$ \\
		\MT & ($m=2$)
		& 40 & 61 & 40 & 25 & $9$ & $1\mathrm{e}{-3}$ & $0.9991$/$0.16$ & $2$/$0.40$ & $4$/$0.42$ \\
		\hline
		Armijo & $(m=3)$ 
		& 46 & 89 & 45 & 29 & 1 & $1\mathrm{e}{-3}$ &  $0.9992$/$0.77$ & $2$/$2$ & $8$/$0.75$ \\
		\MT & ($m=3$)
		& 43 & 65 & 43 & 27 & $21$ & $1\mathrm{e}{-3}$ & $0.9992$/$0.017$ & $9$/$0.18$ & $9$/$0.11$ \\
		\hline
		Armijo & $(m=4)$ 
		& 60 & 114 & 59 & 39 & $1$ & $1\mathrm{e}{-3}$ & $0.9992$/$0.77$ & $12$/$0.88$ & $9$/$0.88$ \\
		\MT & ($m=4$)
		& 51 & 73 & 51 & 33 & $5$ & $1\mathrm{e}{-3}$ & $0.9992$/$0.036$ & $3$/$0.063$ & $6$/$0.34$ \\
	\end{tabular}
	\caption{Results of \ref{alg_hybrid} for the Rosenbrock function. Here, \#it provides the number of iterations, \#$f$ the number of evaluations of $f$, 
		$\CP$ the number of iterations in which the storage is updated, 
		and $\alpha=1$ the number of iterations with $\alpha_k=1$. 
		Since Algorithm~\ref{alg_hybrid} agrees with Algorithm~\ref{alg_LBFGS} in this example, the results for standard \LBFGS~($m>0$) and the \BBB~method ($m=0$) are identical to the ones depicted in the table.}
	\label{tab_TP1}
\end{table}

\Cref{tab_TP1} shows, among others, that using $m=0$ with a strong Wolfe--Powell line search can be disastrous; the convergence is comparable to that of steepest descent (not shown). Indeed, it is well understood that monotone line searches can be too restrictive for the \BB~method, cf. for instance the convergence analysis in \cite{AK20}. For this reason, the \BBB~method is typically combined with a non-monotone line search \cite{R97,GS02,DF05}. We thus applied \ref{alg_hybrid} for $m=0$ with the non-monotone Armijo line search of Grippo et al. \cite{GLL86}, yielding an improvement over 
monotone Armijo. 
Specifically, with the best choice of parameters, 71 iterations and 82 evaluations of $f$ are required.
A further observation in \Cref{tab_TP1} is that for $m>0$ the q-factor is usually significantly smaller during the final 3 iterations than in the previous iterations, suggesting an acceleration in convergence; \Cref{fig_TP1} confirms this observation. 

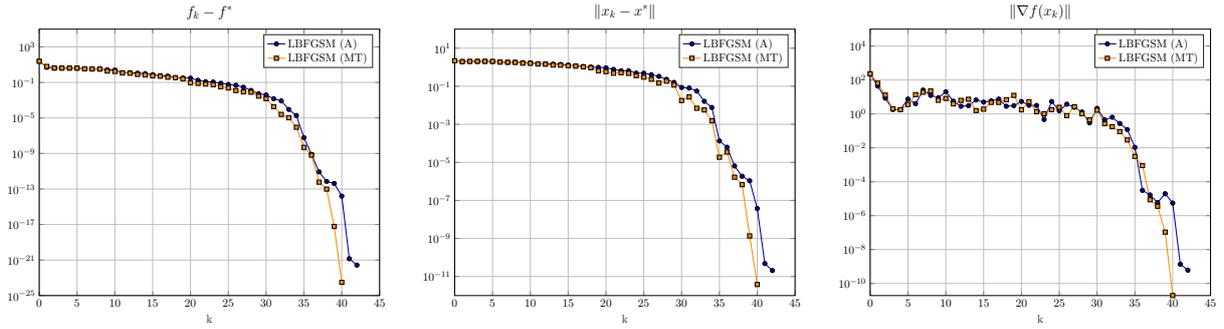
\begin{figure}
	\centering
	\scalebox{0.39}{
		\definecolor{mycolor1}{rgb}{1.00000,0.55000,0.00000}
		\begin{tikzpicture}
			\begin{axis}[
				width=4.521in,
				height=3.566in,
				at={(0.758in,0.481in)},
				scale only axis,
				xmin=0,
				xmax=45,
				xlabel style={font=\color{white!15!black}},
				xlabel={k},
				ymode=log,
				ymin=1e-25,
				ymax=100000,
				yminorticks=true,
				axis background/.style={fill=white},
				title style={font=\bfseries},
				title={\Large $f_k-f^\ast$},
				xmajorgrids,
				ymajorgrids,
				yminorgrids,
				legend style={legend cell align=left, align=left, draw=white!15!black}
				]
				\addplot [color=black!35!blue, line width=1pt, mark=*, mark options={solid, fill=black!35!blue, draw=black}]
				table[row sep=crcr]{
					0	24.2\\
					1	5.10111266371095\\
					2	4.15378842726836\\
					3	4.11721503664523\\
					4	4.11381682449055\\
					5	3.99779331500671\\
					6	3.41929207690587\\
					7	3.29230084302145\\
					8	3.22015480506825\\
					9	2.62327925577181\\
					10	2.47384884410594\\
					11	1.28389061169424\\
					12	1.2592999970055\\
					13	1.10261581687294\\
					14	0.991079874445112\\
					15	0.736431805326911\\
					16	0.553034893921661\\
					17	0.522107753707644\\
					18	0.347914596652493\\
					19	0.319521474559929\\
					20	0.303289893178952\\
					21	0.174905631518707\\
					22	0.120055284731831\\
					23	0.111403497025955\\
					24	0.0802509081516265\\
					25	0.0584822510930375\\
					26	0.0476765749985289\\
					27	0.0288618339483577\\
					28	0.0126009149323327\\
					29	0.00542548336769723\\
					30	0.003963879340115\\
					31	0.00144095700061454\\
					32	0.000842098796687885\\
					33	8.99156894456188e-05\\
					34	1.88271690231954e-05\\
					35	5.99389980575334e-08\\
					36	7.78559785979648e-10\\
					37	8.3987634523441e-12\\
					38	7.01342941471587e-13\\
					39	4.24634672336428e-13\\
					40	1.54692297275873e-14\\
					41	1.40610019566489e-21\\
					42	2.68826256288027e-22\\
				};
				\addlegendentry{LBFGSM (A)}
				\addplot [color=mycolor1, line width=1pt, mark=square*, mark options={solid, fill=mycolor1, draw=black}]
				table[row sep=crcr]{
					0	24.2\\
					1	6.32525573129495\\
					2	4.19633029873649\\
					3	4.11491338388386\\
					4	4.1113289433646\\
					5	4.09213404158842\\
					6	3.37016426416014\\
					7	3.2769509337776\\
					8	3.19460384231008\\
					9	1.90405996754526\\
					10	1.73109468926301\\
					11	1.16844567300691\\
					12	1.08773968277945\\
					13	0.799376305695221\\
					14	0.689684346025524\\
					15	0.548399742081302\\
					16	0.499767660206868\\
					17	0.388265299925242\\
					18	0.338731391008586\\
					19	0.250105576091097\\
					20	0.0898876680251974\\
					21	0.0767939568368226\\
					22	0.0648937307287435\\
					23	0.0548271475798122\\
					24	0.0319079697486555\\
					25	0.0239306619117538\\
					26	0.012129031985988\\
					27	0.00861995708652493\\
					28	0.00790948368668645\\
					29	0.00287402943071791\\
					30	0.00147521128108276\\
					31	0.000187394288165683\\
					32	2.53516520408067e-05\\
					33	1.05726138526238e-05\\
					34	8.95568491967773e-07\\
					35	4.84439403070859e-09\\
					36	6.45695018563336e-10\\
					37	5.88383789748604e-13\\
					38	9.6825504959948e-14\\
					39	6.05683262769979e-18\\
					40	3.07496192359567e-24\\
				};
				\addlegendentry{LBFGSM (MT)}
			\end{axis}
		\end{tikzpicture}
	}
	\hfill
	\scalebox{0.39}{
		\definecolor{mycolor1}{rgb}{1.00000,0.55000,0.00000}
		\begin{tikzpicture}	
			\begin{axis}[
				width=4.521in,
				height=3.566in,
				at={(0.758in,0.481in)},
				scale only axis,
				xmin=0,
				xmax=45,
				xlabel style={font=\color{white!15!black}},
				xlabel={k},
				ymode=log,
				ymin=1e-12,
				ymax=100,
				yminorticks=true,
				axis background/.style={fill=white},
				title style={font=\bfseries},
				title={\Large $\|x_k-x^\ast\|$},
				xmajorgrids,
				ymajorgrids,
				yminorgrids,
				legend style={legend cell align=left, align=left, draw=white!15!black}
				]
				\addplot [color=black!35!blue, line width=1pt, mark=*, mark options={solid, fill=black!35!blue, draw=black}]
				table[row sep=crcr]{
					0	2.2\\
					1	1.99130836147456\\
					2	2.02236413755431\\
					3	2.02886702301742\\
					4	2.02764380640284\\
					5	1.98444013233506\\
					6	1.87017435920913\\
					7	1.78897580155473\\
					8	1.82628362173105\\
					9	1.7342858647109\\
					10	1.67032649078575\\
					11	1.50104295405528\\
					12	1.4962796170711\\
					13	1.45211116475959\\
					14	1.39389002689362\\
					15	1.28830048620941\\
					16	1.16008593999155\\
					17	1.0895693943888\\
					18	1.01189602611497\\
					19	0.978033246873187\\
					20	0.927858638914651\\
					21	0.765247891567089\\
					22	0.649793163272342\\
					23	0.64917165784008\\
					24	0.489684851843272\\
					25	0.483367084029448\\
					26	0.403051369335357\\
					27	0.330979312525608\\
					28	0.230797162172889\\
					29	0.15932120693653\\
					30	0.0863250890681971\\
					31	0.0805890697637356\\
					32	0.0557122056953063\\
					33	0.0161669702083034\\
					34	0.00763893059666378\\
					35	0.000133755170805377\\
					36	6.24349071396726e-05\\
					37	6.4333309869341e-06\\
					38	1.85119086761766e-06\\
					39	1.07744356592602e-06\\
					40	3.74927654764878e-08\\
					41	4.81791795120844e-11\\
					42	2.06353500615171e-11\\
				};
				\addlegendentry{LBFGSM (A)}
				\addplot [color=mycolor1, line width=1pt, mark=square*, mark options={solid, fill=mycolor1, draw=black}]
				table[row sep=crcr]{
					0	2.2\\
					1	1.97075077260364\\
					2	2.01779451428321\\
					3	2.02861525927003\\
					4	2.02690953807918\\
					5	2.01804783285175\\
					6	1.85098763826329\\
					7	1.8183486264428\\
					8	1.78850787220779\\
					9	1.62804810593964\\
					10	1.59600037285774\\
					11	1.47181643385262\\
					12	1.43499846682005\\
					13	1.29320659463801\\
					14	1.27982794719891\\
					15	1.19222127288571\\
					16	1.13319930442152\\
					17	1.03302726167718\\
					18	0.943901192800317\\
					19	0.638012316742387\\
					20	0.584639598674546\\
					21	0.481898445773723\\
					22	0.50832921950867\\
					23	0.472914227950492\\
					24	0.360198328180376\\
					25	0.299463326653393\\
					26	0.23231372163334\\
					27	0.148466577455089\\
					28	0.184785123436609\\
					29	0.115306161915432\\
					30	0.0177463551261582\\
					31	0.0274123066282381\\
					32	0.0069626618921099\\
					33	0.00568207092124757\\
					34	0.00152982299522676\\
					35	1.85927610347522e-05\\
					36	3.48152001640107e-05\\
					37	1.6620155730315e-06\\
					38	6.72958636866253e-07\\
					39	1.36010621660648e-09\\
					40	3.80395310700029e-12\\
				};
				\addlegendentry{LBFGSM (MT)}
			\end{axis}
		\end{tikzpicture}
	}
	\hfill
	\scalebox{0.39}{
		\definecolor{mycolor1}{rgb}{1.00000,0.55000,0.00000}
		\begin{tikzpicture}
			\begin{axis}[
				width=4.521in,
				height=3.566in,
				at={(0.758in,0.481in)},
				scale only axis,
				xmin=0,
				xmax=45,
				xlabel style={font=\color{white!15!black}},
				xlabel={k},
				ymode=log,
				ymin=1.93442497601161e-11,
				ymax=100000,
				yminorticks=true,
				axis background/.style={fill=white},
				title style={font=\bfseries},
				title={\Large $\|\nabla f(x_k)\|$},
				xmajorgrids,
				ymajorgrids,
				yminorgrids,
				legend style={legend cell align=left, align=left, draw=white!15!black}
				]
				\addplot [color=black!35!blue, line width=1pt, mark=*, mark options={solid, fill=black!35!blue, draw=black}]
				table[row sep=crcr]{
					0	232.867687754227\\
					1	43.8985209232249\\
					2	8.46978756375433\\
					3	1.82532681335941\\
					4	1.77476204348877\\
					5	7.48350563094842\\
					6	4.00066636700262\\
					7	25.4848945441268\\
					8	12.3542246281225\\
					9	8.85933461374664\\
					10	20.2608644705895\\
					11	5.61882715933337\\
					12	2.82044085839459\\
					13	3.06929406979361\\
					14	6.64355738577405\\
					15	5.00746269787819\\
					16	5.86476652318916\\
					17	7.4908793490463\\
					18	2.83198096970534\\
					19	3.06266922954256\\
					20	5.44602614474888\\
					21	3.22788050273954\\
					22	3.15338030675841\\
					23	0.467491716882428\\
					24	5.42675231615367\\
					25	1.52873308443467\\
					26	3.74934070856063\\
					27	2.49185036397093\\
					28	1.29914886940003\\
					29	0.298806709976806\\
					30	2.1414686988539\\
					31	0.448141473136954\\
					32	0.636132488078867\\
					33	0.27240957133411\\
					34	0.119347092363002\\
					35	0.0106278984142371\\
					36	3.08552148344568e-05\\
					37	1.66127561309811e-05\\
					38	5.89314473202273e-06\\
					39	1.96625148651745e-05\\
					40	5.51706932738594e-06\\
					41	1.3745295401199e-09\\
					42	6.06960253007594e-10\\
				};
				\addlegendentry{LBFGSM (A)}
				\addplot [color=mycolor1, line width=1pt, mark=square*, mark options={solid, fill=mycolor1, draw=black}]
				table[row sep=crcr]{
					0	232.867687754227\\
					1	64.7722716718182\\
					2	12.8227963861207\\
					3	2.01301101950666\\
					4	1.77714632827451\\
					5	3.59447220957523\\
					6	13.7747826251693\\
					7	18.9989402748628\\
					8	22.4467796919006\\
					9	6.35974264635293\\
					10	8.21732739604537\\
					11	3.90718510268854\\
					12	6.29340123970513\\
					13	7.3548104276944\\
					14	1.58011954331079\\
					15	1.90809228826357\\
					16	4.78124349680322\\
					17	4.77958774020289\\
					18	6.91328529696125\\
					19	12.4501259591309\\
					20	1.75010510071367\\
					21	5.26785150763035\\
					22	1.35105125200025\\
					23	1.0202168154446\\
					24	1.778245388107\\
					25	2.45539486345997\\
					26	0.787899588960235\\
					27	2.67584595407174\\
					28	1.02617263155891\\
					29	0.441605047785429\\
					30	1.67354856217282\\
					31	0.265609133682724\\
					32	0.176686125385063\\
					33	0.0906357232776311\\
					34	0.0292774786924001\\
					35	0.00309355097540646\\
					36	0.000899503392002192\\
					37	8.6140622950244e-06\\
					38	3.59011583940957e-06\\
					39	1.06761121224269e-07\\
					40	1.93442497601161e-11\\
				};
				\addlegendentry{LBFGSM (MT)}
			\end{axis}
		\end{tikzpicture}
	}
	\caption{\Ref{alg_hybrid} with $m=2$ for the Rosenbrock function. The results for Algorithm~\ref{alg_LBFGS} with $m=2$ are identical.}
	\label{fig_TP1}
\end{figure}

Next we comment on \Cref{fig_TP1}. Since $f$ is strongly convex for any $x\in\R^2$ with $\norm{x-\xopt}\leq 0.4$, 
the error plot for $\norm{x_k-\xopt}$ in \Cref{fig_TP1} in combination with \Cref{thm_linconv} indicates that $l$-step q-linear convergence of $(x_k)$ and $(\nabla f(x_k))$ is guaranteed from around iteration $k_1=k_2=26$ onward for Armijo and from $k_1=k_2=24$ onward for \MT; cf. also \Cref{rem_thmlinconv}~2).
In this regard we note that for $k\leq k_1$ the plot of $(\norm{\nabla f(x_k)})$ contains many pairs $(x_{k-l},x_k)$ with $\norm{\nabla f(x_{k-l})}\leq \norm{\nabla f(x_k)}$ for several $l$,
ruling out $l$-step q-linear convergence, while for $k\geq k_1$ we presumably see $l$-step q-linear convergence for any $l\geq 4$ for Armijo, respectively, for any $l\geq 5$ in case of \MT. 
Yet, $5$-step q-linear convergence is violated for 
$k=k_1-1$ since $\norm{\nabla f(x^{30})}>\norm{\nabla f(x^{25})}$ for Armijo and $\norm{\nabla f(x^{28})}>\norm{\nabla f(x^{23})}$ for \MT, where the inequalities are based on the numerical values underlying the figure.

\subsection{Example~2: A piecewise quadratic function}

To demonstrate that \ref{alg_hybrid} works on objectives that are $C^{1,1}$ but not $C^2$, we let $N\in\N$, $d:=3N$ and consider the piecewise quadratic function $f:\R^d\rightarrow\R$, $f(x):=\frac12 \norm{x-b}^2+\frac{99}{2}\sum_{i=1}^{d}\max\{0,x_i\}^2$, 
where $b = (f,f,\ldots,f)^T\in\R^d$ with $f=(1,-1,0)$. 
This objective is strongly convex, every level set is bounded, and it is $C^1$ with $\nabla f(x) = x-b + 99\max\{0,x\}$, where $\max$ is applied componentwise. 
It is clear that $\nabla f$ is Lipschitz in $\R^d$, but not differentiable at $\xopt=(y,y,\ldots,y)^T$ with $y=(0.01,-1,0)$, the unique stationary point of $f$.
As in the first example it follows from \Cref{thm_clusterpointsarestationar} and \Cref{thm_linconv}
that for any starting point Algorithm~\ref{alg_hybrid} either terminates finitely or it generates linearly convergent sequences. 
In contrast to the first example we have $k_1=k_2=0$. In particular, the q-linear convergence estimate \cref{eq_linconvJ} holds for all $k$. 
The convergence behavior of classical~\LBFGS~is exactly the same, but this cannot be inferred from existing results such as \cite[Theorem~7.1]{LN89} since $f$ is not twice continuously differentiable. Instead, it follows from the results of this work, cf. the discussion in \Cref{rem_thmlinconv}~1).

Applying Algorithm~\ref{alg_hybrid} with starting point $x_0=b$ for $N=100$ yields the results displayed in \Cref{tab_TP2}. 
We point out that in this example we do not use the \MT~line search but resort to a line search that ensures the weak \WP~conditions instead. The reason for not using the \MT~line search is that it involves quadratic and cubic interpolation for $f$, which is not appropriate since $f$ is only piecewise smooth. 

\begin{table}
	\footnotesize
	\centering
	\begin{tabular}{cc|c|c|c|c|c|c|c|c|c}
		& & \#it & $\#f$ & $\CP$ & $\alpha=1$ & $\alpha_{\max}$ & $\alpha_{\min}$ & $Q_f/Q_f^3$ & $Q_x/Q_x^3$ & $Q_g/Q_g^3$ \\
		\hline
		Armijo & $(m=0)$ & 10 & 23 & 10 & 3 & $1$ & $0.06$ & $0.88$/$0.13$ & $5.3$/$0.32$ & $3.2$/$0.55$   \\
		\WP & ($m=0$)
		& 8 & 23 & 8 & 1 & 2 & $0.03$ & $0.77$/$0.75$ & $5.3$/$0.87$ & $0.88$/$0.87$\\
		\hline
		Armijo & $(m=5)$ 
		& 11 & 45 & 11 & 2 & 1 & 0.02 & $0.88$/$0.32$ & $5.3$/$0.56$ & $3.2$/$0.56$ \\
		\WP & ($m=5$)
		& 11 & 49 & 11 & 1 & 2 & 0.02 & $0.77$/$0.32$ & $5.3$/$0.56$ & $0.88$/$0.56$ \\
		\hline
		Armijo & $(m=10)$ & 10 & 23 & 10 & 3 & $1$ & $0.06$ & $0.88$/$0.13$ & $5.3$/$0.32$ & $3.2$/$0.55$   \\
		\WP & ($m=10$)
		& 8 & 23 & 8 & 1 & 2 & $0.03$ & $0.77$/$0.75$ & $5.3$/$0.87$ & $0.88$/$0.87$\\
	\end{tabular}
	\caption{Results of \ref{alg_hybrid} for Example~2, a strongly convex and piecewise quadratic objective.
		The results for standard \LBFGS~($m>0$) and the \BBB~method ($m=0$) are identical to those in the table.}
	\label{tab_TP2}
\end{table}

While \Cref{tab_TP2} does not reveal this information, \ref{alg_hybrid} finds the \emph{exact} solution $\xopt$ in the displayed runs. 
It is also interesting that the iterates for $m=0$ and $m=10$ agree if the same line search is used. 
The combination of $m=0$ with a non-monotone Armijo line search (not shown) does not improve the performance in this example. 
The gap between $Q_f$ and 1 is much larger than in Example~1, which we attribute to the fact that \cref{eq_linconvJ} holds for all $k$. 

To check for global convergence we conduct,
for each of the memory sizes and line searches displayed in \Cref{tab_TP2}, 
$10^5$ runs of Algorithm~\ref{alg_hybrid} with random starting points generated by Matlab's \texttt{randn}.
The gradient norm is successfully decreased below $10^{-5}$ in all runs. 
The average number of iterations for $m=0$ is 98.9 for Armijo and 227.7 for \WP, for $m=5$ it is 83.0 for Armijo and 82.4 for \WP, and for $m=10$ it is 92.5 for Armijo and 92.3 for \WP.

\subsection{Example~3: PDE-constrained optimal control}

To illustrate that the results of this work are valid in infinite dimensional Hilbert space, 
we consider a nonconvex large-scale problem from PDE-con\-strained optimal control. 
Recently, \BB-type methods have been applied to and studied for this problem class \cite{DKPS15,LMP21,AK22}.
Numerical studies involving \LBFGS--type methods for PDE-constrained optimal control problems are available in 
\cite{NVM16,CD18,MR21,FVM22}, for instance. Besides the present work, the convergence theory of \LBFGS--type methods in Hilbert space is only addressed in our paper \cite{AMM23}. We consider the problem 
\begin{equation*}
	\min_{u\in L^2(\Omega)} \frac12\norm{y_u-y_d}^2_{L^2(\Omega)} + \frac{\nu}{2}\norm{u}^2_{L^2(\Omega)},
\end{equation*}
where $\Omega=(0,1)^2\subset\R^2$, $y_d\in L^2(\Omega)$, $\nu>0$ and $y=y_u$ denotes the solution to the semilinear elliptic boundary value problem
\begin{equation*}
	\left\{
	\begin{aligned}
		-\Delta y + \exp(y) &= u && \text{ in }\Omega,\\
		y &= 0 && \text{ on }\partial\Omega.
	\end{aligned}\right.
\end{equation*}
It can be shown by standard arguments that for every $u\in L^2(\Omega)=:U$ there is a unique weak solution $y_u\in H_0^1(\Omega)\cap C(\bar\Omega)=:Y$ to this PDE and that the mapping $u\mapsto y_u$ is smooth from $U$ to $Y$, cf. \cite{CRT08} and the references therein. 
Regarding $y_u$ as a function of $u$, the objective $f(u):=\frac12\norm{y_u-y_d}^2_{L^2(\Omega)} + \frac{\nu}{2}\norm{u}^2_{L^2(\Omega)}$
defined on the Hilbert space $U$ is smooth and it admits a global minimizer \cite{CRT08}.
Since $u\mapsto y_u$ is nonlinear, $f$ is nonconvex. 

We choose $\nu=10^{-3}$ and $y_d(x_1,x_2)=\sin(2\pi x_1)\cos(2\pi x_2)$ and
we discretize the Laplacian by the classical 5-point stencil on a uniform grid with $M+1=2^{j}+1$, $4\leq j\leq 11$ points in each direction of the grid. 
The discretization of the control $u$ and the state $y$, denoted $u_h$ and $y_h$, live in $\CX=\R^N$ with $N=(M-1)^2$. Its entries represent function values at the $(M-1)^2$ inner nodes of the grid. 
We obtain the discretization $f_h$ of $f$ by replacing integrals $\int_\Omega v\,\mathrm{d}x$ by $h^2\sum_{i=1}^N [v_h]_i$, where $[v_h]_i$ indicates the $i$-th component of $v_h$ and $h:=1/M$ is the mesh width of the grid. To mimic the $L^2$ inner product, we endow $\R^N$ with the scalar product $(u_h,v_h):=h^2\sum_{i=1}^N ([u_h]_i \cdot [v_h]_i)$. 
Note that this differs from the Euclidean inner product, which has to be taken into account, for instance when computing $B_{k+1}$ (respectively, when applying the two-loop recursion) since formulas \cref{eq_defLBFGSupd2} and \cref{eq_defLBFGSupd} involve the scalar product, cf. also the paragraph \emph{Notation} at the end of Section~\ref{sec_intro}.
From now on all quantities are discrete, so we suppress the index $h$. 
In every iteration we compute the discrete state $y_{u_k}$ associated to the discrete control $u_k$ by a few iterations of a damped Newton's method.
Since the exact solution of the problem is unknown, we run Algorithm~\ref{alg_hybrid} to obtain $u_k$ satisfying $\norm{\nabla f(u_k)}\leq 10^{-12}$ and use it in place of an exact solution. 
The desired state $y_d$, the optimal state $\bar y:=y_{\bar u}$ and the optimal control $\bar u$ are depicted in \Cref{fig_TP3_graphics}. 
The results obtained with starting point $u_0=0$ are displayed in \Cref{tab_TP3} and \Cref{tab_TP3b}. 

\begin{figure}
	\centering
	\includegraphics[width=0.32\linewidth]{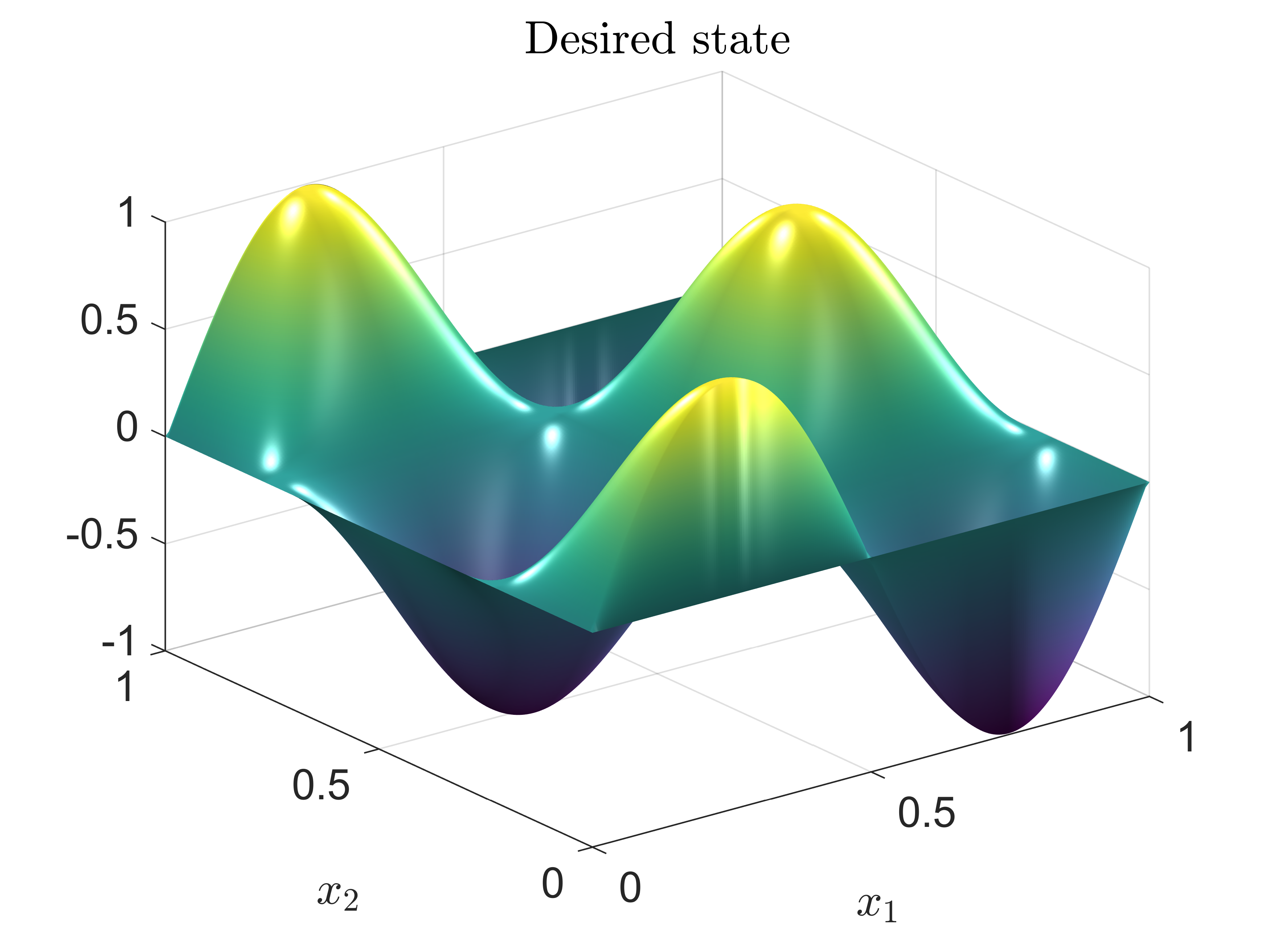}
	\hfill
	\includegraphics[width=0.32\linewidth]{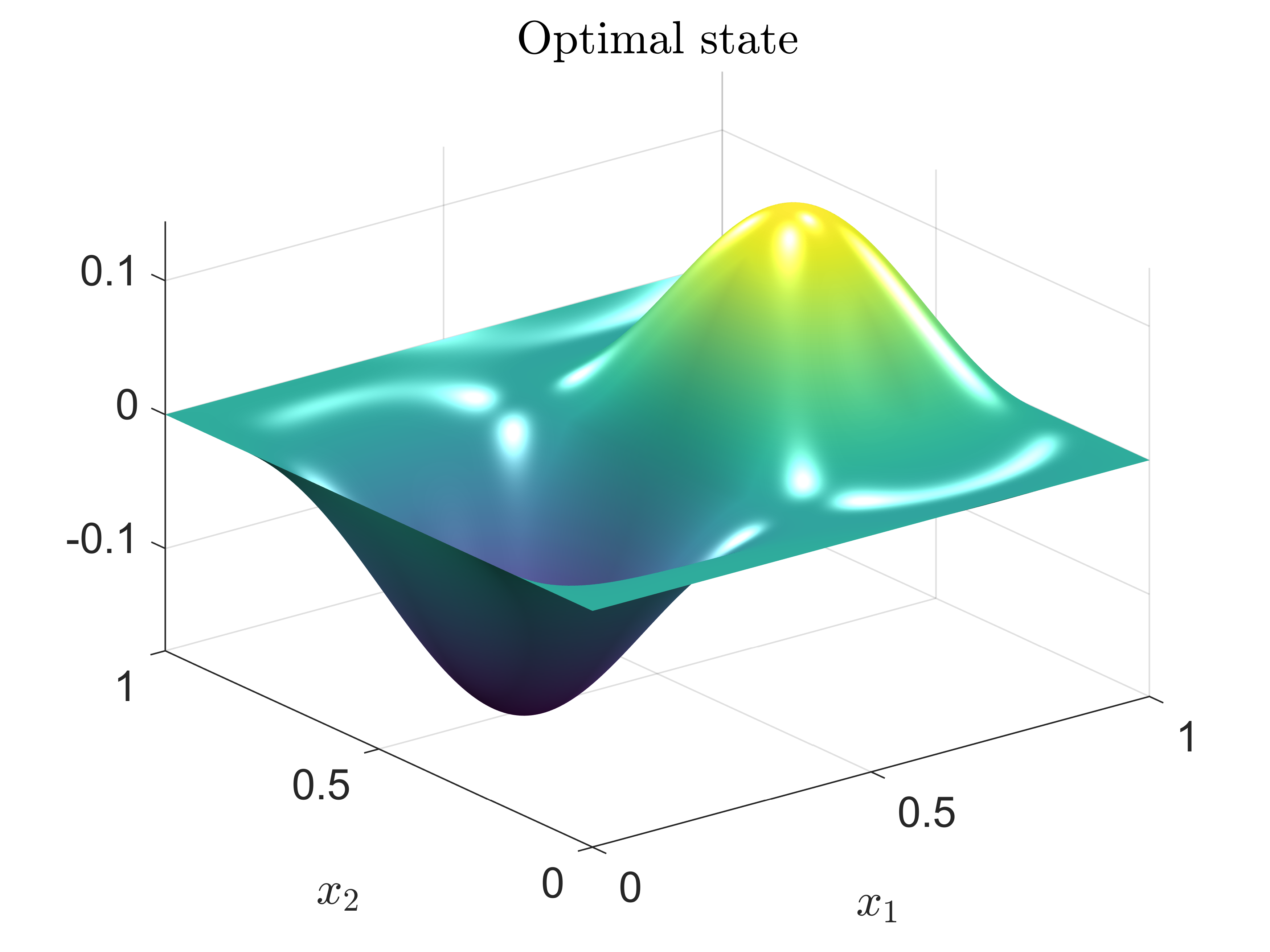}
	\hfill
	\includegraphics[width=0.32\linewidth]{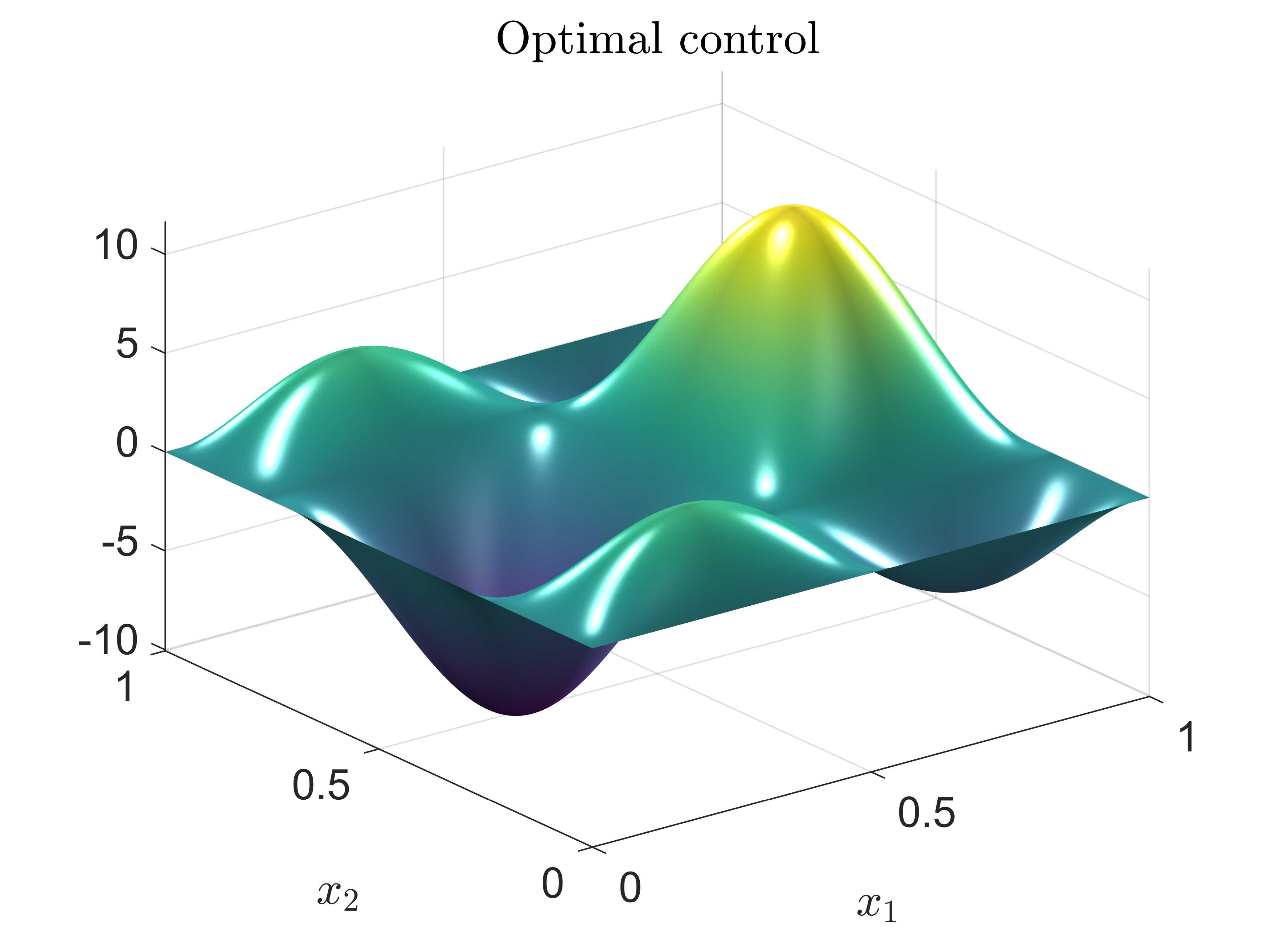}
	\caption{Desired state $y_d$, optimal state $\bar y=y_{\bar u}$, optimal control $\bar u$}
	\label{fig_TP3_graphics}
\end{figure}

\begin{table}
	\scriptsize
	\centering
	\begin{tabular}{cc|c|c|c|c|c|c|c|c|c}
		& & \#it & $\#f$ & $\CP$ & $\alpha=1$ & $\alpha_{\max}$ & $\alpha_{\min}$ & $Q_f/Q_f^3$ & $Q_x/Q_x^3$ & $Q_g/Q_g^3$ \\
		\hline
		Armijo& $(m=0)$ & 14 & 14 & 14 & 14 & $1$ & $1$ & $0.997$/$0.07$ & $0.9987$/$0.26$ & $1.2$/$0.26$ \\
		\MT & ($m=0$)
		& 14 & 17 & 14 & 13 & $85$ & $1$ & $0.784$/$0.12$ & $0.892$/$0.35$ & $0.87$/$0.34$ \\
		\hline
		Armijo & $(m=5)$ 
		& 10 & 10 & 10 & 10 & 1 & 1 & $0.997$/$0.03$ & $0.9987$/$0.17$ & $0.9985$/$0.19$\\
		\MT & ($m=5$)
		& 10 & 13 & 10 & 9 & $85$ & $1$ & $0.784$/$0.03$ & $0.892$/$0.17$ & $0.87$/$0.18$ \\
		\hline
		Armijo & $(m=10)$ 
		& 8 & 8 & 8 & 8 & 1 & $1$ & $0.997$/$0.02$ & $0.9987$/$0.16$ & $0.9985$/$0.15$\\
		\MT & ($m=10$)
		& 8 & 11 & 8 & 7 & 85 & $1$ & $0.784$/$0.03$ & $0.892$/$0.16$ & $0.87$/$0.15$ \\
	\end{tabular}
	\caption{Results of \ref{alg_hybrid} (and simultaneously standard \LBFGS/\BBB) for Example~3, a nonconvex optimal control problem} 
	\label{tab_TP3}
\end{table}

\Cref{tab_TP3} shows, for instance, that the full step is taken in all iterations for the Armijo line search and in all but one iteration for the \MT~line search. A closer inspection reveals that only the first iteration does not use a full step. 
Although the problem is nonconvex, we have $y^T s>0$ in all iterations, cf. the values of $\CP$ in \Cref{tab_TP3}.
For $m=10$ we have $\text{iter}<m$, so any choice $m>10$ produces exactly the same results as for $m=10$.
The values of $Q_f$, $Q_x$ and $Q_g$ indicate that for the \MT~line search we have 
q-linear convergence for the objective values, the iterates and also the gradients. 
Similar to Example~1, the last three columns of \Cref{tab_TP3} hint at the fact that a significant acceleration takes place during the course of the algorithm.

An important property of efficient numerical algorithms for PDE-con\-strained optimization is their \emph{mesh independence} \cite{ABPR86,KS87,AK22}. This roughly means that the number of iterations to reach a prescribed tolerance is insensitive to the mesh size. 
\Cref{tab_TP3b} clearly confirms the mesh independence of Algorithm~\ref{alg_hybrid}. 
Finally, we mention that in this example we used $\sigma=10^{-8}$ for the \MT~line search since $\sigma=10^{-4}$ sometimes failed. 

\begin{table}
	\small
	\centering
	\begin{tabular}{cc|c|c|c|c|c|c|c|c}
		& $M=2^j, j=$ & $4$ & $5$ & $6$ & $7$ & $8$ & $9$ & $10$ & $11$ \\
		\hline
		Armijo & $(m=0)$ & 15 & 14 & 14 & 14 & 14 & 14 & 14 & 14 \\
		\MT & ($m=0$)
		& 15 & 15 & 14 & 14 & 14 & 14 & 14 & 14 \\
		\hline
		Armijo & $(m=5)$ & 10 & 10 & 10 & 10 & 10 & 10 & 10 & 10 \\
		\MT & ($m=5$)
		& 10 & 10 & 10 & 10 & 10 & 10 & 10 & 10 \\
		\hline
		Armijo & $(m=10)$ & 8 & 8 & 8 & 8 & 8 & 8 & 8 & 8 \\
		\MT & ($m=10$)
		& 8 & 8 & 8 & 8 & 8 & 8 & 8 & 8
	\end{tabular}
	\caption{Iteration numbers of \ref{alg_hybrid} (and standard \LBFGS/\BBB) for Example~3 with different mesh sizes. 
		The discretized control belongs to $\CX=\R^N$ with $N=(M-1)^2$. 
		The iteration numbers in each row are practically independent of the mesh size, 
		which is reflective of the fact that the convergence results for \ref{alg_hybrid} hold in $L^2(\Omega)$}
	\label{tab_TP3b}
\end{table}


\section{Conclusion}\label{sec_conclusion}

This work introduces the first globally convergent modification of \LBFGS~that recovers classical \LBFGS~under locally sufficient optimality conditions. 
The strong convergence guarantees of the method rely on several modifications of cautious updating, including the novel idea to decide in each iteration, based on the most current gradient norm only, which storage pairs to use and which ones to skip. 
The method enjoys q-linear convergence for the objective values and $l$-step q-linear convergence for the iterates and the gradients for all sufficiently large $l$. 
The rates of convergence rely on strong convexity and gradient Lipschitz continuity, but only near one cluster point of the iterates; the existence of second derivatives is not required. Global convergence for \WP~line searches is shown without boundedness of the level set and using only continuity of the gradient. 
The rates are also valid for classical \LBFGS~if its iterates converge to a point near which the objective is strongly convex with Lipschitz continuous gradient, which is for instance satisfied if the objective is strongly convex in the level set and has Lipschitz gradients near its unique stationary point.
The results hold in any Hilbert space and also for memory size $m=0$, yielding a new globalization of the \BB~method. 
Numerical experiments support the theoretical findings and suggest that for sufficiently small parameter $c_0$, the new method and \LBFGS~agree entirely. 
This may explain why \LBFGS~is often successful for nonconvex problems. 
		
%
%

\bibliographystyle{alphaurl}
\bibliography{lit}

\end{document}